\documentclass[a4paper,11pt]{amsart}

\usepackage{amsmath, amsfonts, amsthm, amssymb, amscd, mathtools, tensor}
\usepackage{graphicx}
\usepackage[left=3.4cm, right=3.4cm]{geometry}
\usepackage{amsaddr}
\usepackage{pst-grad}
\usepackage{url}
\usepackage{pstricks}
\usepackage{varwidth}
\usepackage{bussproofs}
\usepackage[utf8]{inputenc}
\usepackage{pstricks-add}
\usepackage{pgf}
\usepackage{tikz-cd}
\usepackage[all]{xy}
\usepackage{pinlabel}
\usepackage{psfrag}
\usepackage{hyperref}
\usepackage{float}
\usepackage{color}
\usepackage{tikz}
\allowdisplaybreaks

\newtheorem{thm}{Theorem}[section]
\newtheorem{prop}[thm]{Proposition}
\newtheorem{lem}[thm]{Lemma}
\newtheorem{cor}[thm]{Corollary}
\theoremstyle{definition}
\newtheorem{rem}[thm]{Remark}
\newtheorem{definition}{Definition}[section]
\numberwithin{equation}{section}

\newcommand\mycom[2]{\genfrac{}{}{0pt}{}{#1}{#2}}
\newcommand{\HDS}{\vrule width0pt height2.3ex depth1.05ex\displaystyle}
\newcommand{\f}[2]{{\frac{\HDS #1}{\HDS #2}}}
\newcommand{\afrac}[1]{\mycom{\phantom{\HDS #1}}{\HDS #1}}

\newcommand\strt{\stackrel{\textbf{.}\,}{\rightarrow}}

\newcommand{\fp}[2]{{\mycom{\HDS #1}{\HDS #2}}}

\newcommand\rza{{\mbox{\hspace{0.5em}}}}
\newcommand\rzb{{\mbox{\hspace{2em}}}}

\newcommand\otm{\otimes}
\newcommand\str{\rightarrow}
\newcommand\mj{\mbox{\bf 1}}
\newcommand\mc{\mbox{\bf c}}
\newcommand\kon{\otimes}

\newcommand{\pravilo}[1]{ \rza \makebox[-.5em][l]{\mbox{\rm #1}}}

\makeatletter
\renewcommand{\email}[2][]{%
  \ifx\emails\@empty\relax\else{\g@addto@macro\emails{,\space}}\fi%
  \@ifnotempty{#1}{\g@addto@macro\emails{\textrm{(#1)}\space}}%
  \g@addto@macro\emails{#2}%
}
\makeatother

\def\d#1{{#1\kern-0.4em\char"16\kern-0.1em}}
\def\D#1{{\raise0.2ex\hbox{-}\kern-0.4em#1}}
\def \Dj{\mbox{\raise0.3ex\hbox{-}\kern-0.4em D}}
\definecolor{britishracingzelena1}{rgb}{0.0, 0.26, 0.15}
\def\zn{,\kern-0.09em,}

\title{Coherence for logicians}

\author[Petri\' c]{Zoran Petri\' c}
\address{\scriptsize{
Mathematical Institute of the Serbian Academy of Sciences and Arts\\ Kneza Mihaila 36, 11000 Belgrade, Serbia}}
\email{zpetric@mi.sanu.ac.rs}
\author[Zeki\' c]{Mladen Zeki\' c}
\address{\scriptsize{
Mathematical Institute of the Serbian Academy of Sciences and Arts\\ Kneza Mihaila 36, 11000 Belgrade, Serbia}}
\email{mzekic@mi.sanu.ac.rs}

\date{}

\begin{document}

\begin{abstract}
This paper is addressed to logicians not familiar with category theory. It gives a new proof of coherence for symmetric monoidal closed categories, proven by Kelly and Mac Lane in the early 1970s. We find this result of great importance for proof theory and it is formulated here in pure logical terminology free of categorial notions. Coherence is related to the generality conjecture in general proof theory and we hope that our formulation will make it closer to the proof-theoretical community.

\vspace{.3cm}


\vspace{.5ex}

\noindent {\small {\it Keywords$\,$}: General proof theory, equality of derivations, symmetric monoidal closed categories}
\end{abstract}

\maketitle
\section{Introduction}

By a traditional viewpoint, the role of logic is to provide a foundation of mathematics. However, logic or logical techniques may help in proving some results specific for the rest of mathematics. For example, model theory or set theory provide results important to algebra, analysis and other fields. There is also a great influence of recursion theory and proof theory on theoretical computer science. The relation of logic to the rest of mathematics, via category theory is briefly explained in \cite[Preface, Historical perspective on Part II]{LS86}.

The main result of this paper has the same mathematical content as the one proven in \cite{KML71}, where Kelly and Mac Lane used a proof-theoretical technique of cut-elimination in order to show a coherence result for symmetric monoidal closed categories. This result should be of certain proof-theoretical interest. (For a discussion on how logicians and categorists accept applications of logic to category theory and vice versa see \cite[Preface]{LS86}.) Our intention is to present this result using purely logical terminology, hence make it closer to the proof-theoretical community.

A new in this manuscript relative to \cite{KML71}, \cite{L69}, \cite{DP04} and \cite{DP07} is the following. It provides a proof-theoretic reformulation of coherence, avoiding categorical language. It separates the properties intrinsic to the logical system (see Theorem~\ref{cutelimS}, and Propositions~\ref{Lemma:Main1} and \ref{Lemma:Main2}) from the algebra of terms coding derivations in the system. Also, an explicit bridge to the concerns of general proof theory, including
the generality conjecture (see Remark~\ref{dodatak6}) is given. The book \cite{DP04} is completely devoted to proof-theoretical coherence results checking the identity of derivations in the category of relations. That category is replaced here by the category of 1-dimensional cobordisms. The book \cite{DP07} takes the main result of this manuscript as granted, hence those two texts complement each other. This paper contains a section translating the obtained proof-theoretical results back into categorial terminology.

Our proof follows the approach to coherence results introduced in \cite{DP04}. In the original proof from \cite{KML71}, cut-elimination is somewhat hidden, being expressed in the context of (generalized) natural transformations. In contrast, our proof explicitly presents the cut-elimination theorem for a specific formal system, in a standard form similar to that in \cite{L68}.

The main contribution of such a coherence result to logic is tied to the field of \emph{general} proof theory, initiated by Prawitz in the early 1970s. This is a part of proof theory whose task is to answer the questions ``What is a proof?'', ``When two formal derivations are equivalent?'' and other related problems. These are intrinsic logical problems and by solving them one does not help the rest of mathematics, namely the opposite, some help of the rest of mathematics is expected in a solution. Such a standpoint considers logic as part of mathematics rather than its meta-theory.

For a survey on general proof theory see \cite{D03}. Do\v sen's suggestion was that the above questions have to be answered in a way analogous to Church's Thesis---no formal proof of it should be expected. In \cite{D03} he discusses two approaches to a solution, namely the normalization conjecture and the generality conjecture. The idea behind the latter approach is to diversify the variables in a derivation as much as possible without changing the rules of inference, in order to obtain its maximal generalization. Two derivations with the same premisses and conclusions are considered equivalent when for every maximal generalization of one of them there exists a maximal generalization of the other such that both have the same premisses and conclusions. 

The notion of coherence in category theory is closely related to the generality conjecture. Simplest coherence results are of the form ``all diagrams commute''. For example, if one assumes that the diagram~\ref{pentagon} of rebracketing tensor product commutes, then all the diagrams concerning associativity of tensor product commute (see \cite{ML63}).

However, more involved examples of coherence check the commutativity of diagrams through a functor whose source is the category where these diagrams live, and whose target is usually a ``geometrical'' category. Here we use a geometrical category whose arrows are links connecting letters in formulae of one derivation, which are forced by the rules to be the same. Also, we show that it is possible to obtain a maximal generalization of a derivation so that two letters in the source or conclusion of a derivation are equal if and only if they are connected by a link. This suffices to conclude that the coherence result (see Section~\ref{application}) implies that the generality conjecture holds for a large fragment of our system $\mathcal{IL}$ introduced in Section~\ref{IL}.

For example, the following two derivations are not equivalent. The coloured paths in these derivations denote the letters forced by the rules to be equal. Hence, every generalization of the derivation at the left-hand side has equal first two letters in the antecedent, while these two letters are distinct in every maximal generalization of the right-hand side derivation. 

\centerline{\includegraphics[width=1.3\textwidth]{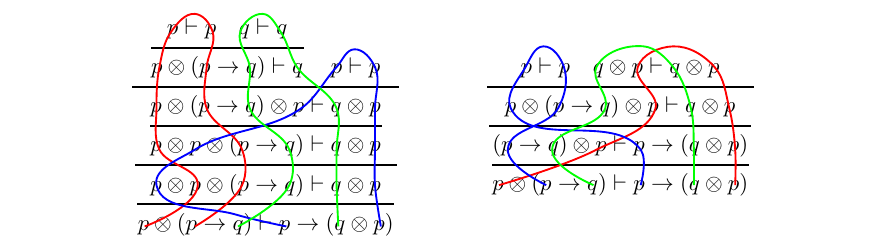} }

Our text consists of three (sometimes overlapping) parts. We start with a purely logical part by introducing a sequent system and proving some of its proof-theoretical properties, which are essential for our main result. In the second part we deal with the ``algebra of proofs'' of such a system. We introduce a language for coding derivations, and define an equational theory on the terms of this language. After proving several properties of this term algebra, a proof-theoretical result corresponding to symmetric monoidal closed coherence is obtained. As a consequence we have that the generality conjecture is \emph{almost} acceptable for our system. This part is ended in showing possible applications, which also make a bridge between viewpoints of logicians and categorists on the main result of the paper.

The last part serves to a logician, not familiar with category theory, to see how some unnamed notions occurring in the main body of the text could be expressed in categorial terminology. Since such  notions are pretty standard and are tied to some classical results on categories, she or he could become more interested in connection between logic and category theory. Moreover, by giving examples of categories satisfying the same conditions as our algebra of proofs, we justify our view of logic as a part of mathematics.

\section{The system \texorpdfstring{$\mathcal{S}$}{S}}

In this section, we introduce a formal system denoted by $\mathcal{S}$. This system corresponds to a fragment of intuitionistic linear logic. However, our notation is categorial and not the standard one used in linear logic. Namely, the constant $I$ that we use corresponds to the multiplicative
constant $\mj$ (the weakening rule given below plays the role of the usual $\mj L$ rule) and our $\str$ corresponds to the linear implication usually denoted by $\multimap$.

The \textit{formulae} of $\mathcal{S}$ are built out of an infinite set of propositional letters and the constant $I$, with the help of two binary logical connectives $\kon$ and $\str$. The \textit{sequents} of $\mathcal{S}$ are of the form $\Gamma\vdash A$, where $\Gamma$ is a sequence (possibly empty) of formulae, and $A$ is a formula. We call $\Gamma$ in $\Gamma\vdash A$ the \textit{antecedent}, and $A$ the \textit{consequent} of the sequent. The \textit{axioms} of $\mathcal{S}$ are
$$A\vdash A \quad \text{and}\quad \vdash I.$$

\noindent The \textit{structural inference figures} of $\mathcal{S}$ are
\begin{prooftree}
\AxiomC{$\Gamma\vdash A$}
\RightLabel{weakening}
\UnaryInfC{$I,\Gamma\vdash A$}
\DisplayProof \qquad
\AxiomC{$\Gamma,A,B,\Delta\vdash C$}
\RightLabel{interchange}
\UnaryInfC{$\Gamma,B,A,\Delta\vdash C$}
\end{prooftree}

\begin{prooftree}
\AxiomC{$\Gamma\vdash A$}
\AxiomC{$\Delta,A,\Theta\vdash B$}
\RightLabel{cut}
\BinaryInfC{$\Delta,\Gamma,\Theta\vdash B$}
\end{prooftree}

\noindent The \textit{operational inference figures} of $\mathcal{S}$ are

\begin{prooftree}
\AxiomC{$\Gamma,A,B,\Delta\vdash C$}
\RightLabel{$\otimes\vdash$}
\UnaryInfC{$\Gamma,A\otimes B,\Delta\vdash C$}
\DisplayProof \qquad
\AxiomC{$\Gamma\vdash A$}
\AxiomC{$\Delta\vdash B$}
\RightLabel{$\vdash\otimes$}
\BinaryInfC{$\Gamma,\Delta\vdash A\otimes B$}
\end{prooftree}

\begin{prooftree}
\AxiomC{$\Gamma\vdash A$}
\AxiomC{$B,\Delta\vdash C$}
\RightLabel{$\to\vdash$}
\BinaryInfC{$\Gamma,A\to B,\Delta\vdash C$}
\DisplayProof \qquad
\AxiomC{$A,\Gamma\vdash B$}
\RightLabel{$\vdash\to$}
\UnaryInfC{$\Gamma\vdash A\to B$}
\end{prooftree}

\begin{rem}\label{atomize}
A sequent $\Gamma,A\otm B,\Theta\vdash C$ has a (cut-free) derivation in $\mathcal{S}$ if and only if $\Gamma,A, B,\Theta\vdash C$ has a (cut-free) derivation in $\mathcal{S}$. The direction from left-to-right is proved by induction on complexity of a derivation of the sequent $\Gamma,A\otm B,\Theta\vdash C$, and the other direction is straightforward.
\end{rem}

\begin{definition}
We say that a formula is \textit{constant} if it does not contain propositional letters (e.g., $(I\str(I\kon I))\kon(I\str I)$ is constant). A sequence is \emph{constant} if it contains only constant formulae.
\end{definition}

By induction on complexity of $A$ we can easily prove the following lemma. (For a strengthening of this result see Lemma~\ref{AtoI}.)

\begin{lem} \label{Lemma:ConstI}
If $A$ is a constant formula, then $A\vdash I$ and $I\vdash A$ are derivable in $\mathcal{S}$.
\end{lem}

The following is a straightforward corollary of Lemma~\ref{Lemma:ConstI}.

\begin{cor} \label{Corollary:ConstConst}
If $\Gamma$ is a constant sequence and $A$ is a constant formula, then $\Gamma\vdash A$ is derivable in $\mathcal{S}$.
\end{cor}

As it is expected, a cut-elimination procedure is applicable to the system $\mathcal{S}$. A strengthening of such a procedure will be completely described later in the proof of Theorem~\ref{cutelim1} in Appendix~\ref{sec:appendix}, which, together with Proposition~\ref{translation}, delivers the following result.

\begin{thm}[Cut-elimination]\label{cutelimS}
Every derivation in $\mathcal{S}$ can be transformed into a cut-free derivation of the same sequent.
\end{thm}

At the end of this section, we prove a lemma that is essentially a corollary of Theorem~\ref{cutelimS}. We need the following definition.

\begin{definition}
We say that a formula is \textit{proper} if it does not contain subformulae of the form $B\str C$, where $C$ is constant and $B$ is not constant (e.g., $(I\str I)\str (I\str p)$ is proper, while $(p\kon I)\str(I\str I)$ is not). A sequence $\Gamma$ is \emph{proper} if all the formulae contained in $\Gamma$ are proper. Finally, a sequent $\Gamma\vdash A$ is \emph{proper} if $\Gamma$ is a proper sequence, and $A$ is a proper formula.
\end{definition}

\begin{lem} \label{Lemma:ConstProper}
Suppose that $\Gamma\vdash A$ is derivable in $\mathcal{S}$, where $\Gamma$ is proper and $A$ is constant. Then $\Gamma$ is constant.
\end{lem}
\begin{proof}
We proceed by induction on the number of occurrences of $\kon$, $\str$ and $I$ in the sequent $\Gamma\vdash A$. The base case is when $\Gamma$ is empty and $A$ is $I$ and it is trivial. For the induction step we consider five cases, depending on the last inference figure in a cut-free derivation of $\Gamma\vdash A$.

\vspace{1ex}
\noindent\textit{Case 1.}  Suppose that our derivation ends with weakening. Then $\Gamma$ is of the form $I,\Delta$, and by the induction hypothesis applied to $\Delta\vdash A$ we have that $\Delta$ is constant, and hence $\Gamma$ is constant.

\vspace{1ex}
\noindent\textit{Case 2.}  If our derivation ends with $\kon\vdash$:
\begin{prooftree}
\AxiomC{$\Delta,B,C,\Theta\vdash A$}
\RightLabel{,}
\UnaryInfC{$\Delta,B\kon C,\Theta \vdash A$}
\end{prooftree}
then by the induction hypothesis we have that $\Delta,B,C,\Theta$ is constant, and consequently $\Gamma$ is also constant.

\vspace{1ex}
\noindent\textit{Case 3.}  If our derivation ends with $\vdash\kon$:
\begin{prooftree}
\AxiomC{$\Delta\vdash A_1$}
\AxiomC{$\Theta\vdash A_2$}
\RightLabel{,}
\BinaryInfC{$\Delta,\Theta\vdash A$}
\end{prooftree}
then, since $A$ and hence $A_1$ and $A_2$ are constant, by the induction hypothesis we have that $\Delta$ and $\Theta$ are constant. Therefore, $\Gamma$ is also constant.

\vspace{1ex}
\noindent\textit{Case 4.}  If our derivation ends with $\str\vdash$:
\begin{prooftree}
\AxiomC{$\Delta\vdash B$}
\AxiomC{$C,\Theta\vdash A$}
\RightLabel{,}
\BinaryInfC{$\Delta, B\str C,\Theta\vdash A$}
\end{prooftree}
then by the induction hypothesis applied to the right premise, we have that $C,\Theta$ is constant. In particular, $C$ is a constant formula, and since $\Gamma$ is proper, we conclude that $B$ is also constant. Now we can apply the induction hypothesis to the left premise as well, so we have that $\Delta$ is constant. Thus, $\Gamma$ is a constant sequence.

\vspace{1ex}
\noindent\textit{Case 5.} If our derivation ends with $\vdash\str$:
\begin{prooftree}
\AxiomC{$A_1,\Gamma\vdash A_2$}
\RightLabel{,}
\UnaryInfC{$\Gamma\vdash A$}
\end{prooftree}
then by the induction hypothesis ($A_2$ is constant, and since $A_1$ is constant, $\Gamma$ is proper, we have that $A_1,\Gamma$ is proper) we conclude that $A_1,\Gamma$ is constant. In particular, $\Gamma$ is constant, which proves the lemma.
\end{proof}

\begin{rem}
Note that we rely on the assumption that $\Gamma$ is proper only in Case~4 of the above proof. However, it is not hard to see that Lemma~\ref{Lemma:ConstProper} is not valid without this assumption. For example, if $\Gamma$ is $p,p\str I$ and $A$ is $I$, then $\Gamma\vdash A$ is derivable in $\mathcal{S}$ and $A$ is constant, but $\Gamma$ is not.
\end{rem}

\section{Two propositions about derivability in $\mathcal{S}$}

In this section we prove two results concerning derivability in $\mathcal{S}$ (Propositions~\ref{Lemma:Main1} and \ref{Lemma:Main2}), which are essential for our main theorem. Both propositions have  flavour of interpolation results, but more appropriate name would be  \emph{splitting of derivations} in this system. We start with the following definition.

\begin{definition}
Let $\Gamma$ and $\Delta$ be two sequences of formulae. We say that $\Gamma$ and $\Delta$ are \textit{disjoint} when there is no propositional letter occurring simultaneously in a formula from $\Gamma$ and a formula from $\Delta$ (e.g., $p\kon (q\str I), r\str p$ and $I\str t$ are disjoint, while $p\kon (q\str I), r\str p$ and $I\str t, q\kon I$ are not). For a sequence $\Gamma$ of formulae, let $\Pi_\Gamma$ denote an arbitrary \emph{permutation} of $\Gamma$.
\end{definition}

\begin{rem} \label{Remark:PiGamma}
Let $\Gamma$ be a sequence of formulae. Then $\Pi_{\Gamma}\vdash A$ is derivable in $\mathcal{S}$ if and only if $\Gamma\vdash A$ is derivable in $\mathcal{S}$ (just apply an appropriate number of interchanges).
\end{rem}

\begin{lem} \label{Lemma:Main0}
Let $\Pi_{\Gamma,\Delta}\vdash A$ be derivable in $\mathcal{S}$ and let $\Delta$ and $\Gamma,A$ be disjoint. Then $\Delta\vdash I$ and $\Gamma\vdash A$ are derivable in $\mathcal{S}$.
\end{lem}
\begin{proof}
We will prove that if $\Pi_{\Gamma,\Delta}\vdash A$ is derivable in $\mathcal{S}$, while $\Delta$ and $\Gamma,A$ are disjoint, then for some $\Pi_\Gamma$ and $\Pi_\Delta$, the sequents $\Pi_\Delta\vdash I$ and $\Pi_\Gamma\vdash A$ are derivable in $\mathcal{S}$. By Remark \ref{Remark:PiGamma} this suffices for our proof. We proceed by induction on complexity of a cut-free derivation of $\Pi_{\Gamma,\Delta}\vdash A$. This complexity could be measured by the height of this derivation, or by the number of sequents in it, or in some other way.

The basis of this induction, i.e.\ when our derivation of $\Pi_{\Gamma,\Delta}\vdash A$ consists just of an axiom, is easy to deal with by relying on Corollary~\ref{Corollary:ConstConst} in some cases when $A$ is constant. For the induction step we have to consider the following cases depending on the last rule applied in our derivation of $\Pi_{\Gamma,\Delta}\vdash A$.

\vspace{2ex}
\noindent\textit{Case 1.} If our derivation ends with weakening, then we apply the induction hypothesis to the premise of this weakening. Moreover, we apply weakening once to the appropriate derived sequent (depending on whether $I$ is introduced within $\Gamma$ or $\Delta$).

\vspace{2ex}
\noindent\textit{Case 2.} If our derivation ends with interchange, then we just apply the induction hypothesis to the premise of this rule.

\vspace{2ex}
\noindent\textit{Case 3.} If our derivation ends with $\kon\vdash$, and  $B\kon C$ is introduced,
then we apply the induction hypothesis to the premise of this rule. Moreover, by relying on Remark~\ref{Remark:PiGamma}, we may assume that in the derived sequent containing $B$ and $C$, these two formulae are consecutive. It remains to apply $\kon\vdash$ to this sequent.

\vspace{2ex}
\noindent\textit{Case 4.} Assume that our derivation ends with $\vdash\kon$:
\begin{prooftree}
\AxiomC{$\Pi_{\Gamma_1,\Delta_1}\vdash A_1$}
\AxiomC{$\Pi_{\Gamma_2,\Delta_2}\vdash A_2$}
\RightLabel{,}
\BinaryInfC{$\Pi_{\Gamma,\Delta}\vdash A$}
\end{prooftree}
where $\Gamma_i$ and $\Delta_j$ are subsequences of $\Gamma$ and $\Delta$, respectively (we keep to this notation in the sequel). By the induction hypothesis we have that $\Pi_{\Delta_1}\vdash I$, $\Pi_{\Gamma_1}\vdash A_1$, $\Pi_{\Delta_2}\vdash I$ and $\Pi_{\Gamma_2}\vdash A_2$ are derivable in $\mathcal{S}$. Applying $\vdash\kon$ to $\Pi_{\Gamma_1}\vdash A_1$ and $\Pi_{\Gamma_2}\vdash A_2$ we obtain that $\Pi_{\Gamma}\vdash A$ is derivable in $\mathcal{S}$, for $\Pi_{\Gamma}=\Pi_{\Gamma_1},\Pi_{\Gamma_2}$. From the following derivation
\begin{prooftree}
\AxiomC{$\Pi_{\Delta_1}\vdash I$}
\AxiomC{$\Pi_{\Delta_2}\vdash I$}
\BinaryInfC{$\Pi_{\Delta_1},\Pi_{\Delta_2}\vdash I\kon I$}
\AxiomC{$I\vdash I$}
\UnaryInfC{$I,I\vdash I$}
\UnaryInfC{$I\kon I\vdash I$}
\RightLabel{,}
\BinaryInfC{$\Pi_{\Delta_1},\Pi_{\Delta_2}\vdash I$}
\end{prooftree}
one concludes that $\Pi_{\Delta}\vdash I$ is derivable in $\mathcal{S}$, for $\Pi_{\Delta}=\Pi_{\Delta_1},\Pi_{\Delta_2}$.

\vspace{2ex}
\noindent\textit{Case 5.} Assume that our derivation ends with $\str\vdash$, and that $\str$ is introduced within $\Gamma$ (the case when it is introduced within $\Delta$ is treated analogously):
\begin{prooftree}
\AxiomC{$\Pi_{\Gamma_1,\Delta_1}\vdash B$}
\AxiomC{$C,\Pi_{\Gamma_2,\Delta_2}\vdash A$}
\RightLabel{.}
\BinaryInfC{$\Pi_{\Gamma,\Delta}\vdash A$}
\end{prooftree}
By the induction hypothesis we have that $\Pi_{\Delta_1}\vdash I$, $\Pi_{\Gamma_1}\vdash B$, $\Pi_{\Delta_2}\vdash I$ and $C,\Pi_{\Gamma_2}\vdash A$ are derivable in $\mathcal{S}$ (for the latter we rely on Remark~\ref{Remark:PiGamma}). Applying $\str\vdash$ to $\Pi_{\Gamma_1}\vdash B$ and $C,\Pi_{\Gamma_2}\vdash A$ we have that $\Pi_{\Gamma}\vdash A$ is derivable in $\mathcal{S}$, for $\Pi_{\Gamma}=\Pi_{\Gamma_1},B\str C,\Pi_{\Gamma_2}$. Also, $\Pi_{\Delta}\vdash I$ is derivable in $\mathcal{S}$ (in the same manner as in Case 4).

\vspace{2ex}
\noindent\textit{Case 6.} Assume that our derivation ends with $\vdash\str$:
\begin{prooftree}
\AxiomC{$A_1,\Pi_{\Gamma,\Delta}\vdash A_2$}
\RightLabel{.}
\UnaryInfC{$\Pi_{\Gamma,\Delta}\vdash A$}
\end{prooftree}
By the induction hypothesis we have that $\Pi_{\Delta}\vdash I$ and $A_1,\Pi_{\Gamma}\vdash A_2$ are derivable in $\mathcal{S}$  (for the latter we rely on Remark~\ref{Remark:PiGamma}), and it remains to apply $\vdash\str$ to $A_1,\Pi_{\Gamma}\vdash A_2$.
\end{proof}

The following two propositions are derived with the help of Lemma~\ref{Lemma:Main0}.

\begin{prop} \label{Lemma:Main1}
If $\Pi_{\Gamma,\Delta}\vdash A\kon B$ is derivable in $\mathcal{S}$, while  $\Gamma,A$ and $\Delta,B$ are disjoint, then $\Gamma\vdash A$ and $\Delta\vdash B$ are derivable in $\mathcal{S}$.
\end{prop}
\begin{proof}
We proceed by induction on complexity of a cut-free derivation of the sequent $\Pi_{\Gamma,\Delta}\vdash A\kon B$ to prove that if $\Gamma,A$ and $\Delta,B$ are disjoint, then for some $\Pi_\Gamma$ and $\Pi_\Delta$, the sequents $\Pi_\Gamma\vdash A$ and $\Pi_\Delta\vdash B$ are derivable in $\mathcal{S}$. By Remark~\ref{Remark:PiGamma}, this suffices for our proof.

The basis of this induction, when our derivation consists just of an axiom, is straightforward. Also, the cases when our derivation ends with weakening, interchange or $\kon\vdash$ are easy to deal with by appealing to the induction hypothesis. It remains to consider the following two cases, in which we use Lemma~\ref{Lemma:Main0}.

\vspace{2ex}
\noindent\textit{Case 1.} Our derivation ends with $\vdash\kon$:
\[
\f{\Pi_{\Gamma_1,\Delta_1}\vdash A \quad \Pi_{\Gamma_2,\Delta_2}\vdash B}{\Pi_{\Gamma,\Delta}\vdash A\kon B.}
\]
By Lemma~\ref{Lemma:Main0} we have that $\Gamma_1\vdash A$, $\Delta_1\vdash I$, $\Gamma_2\vdash I$ and $\Delta_2\vdash B$ are derivable in $\mathcal{S}$, from which one easily concludes that for some $\Pi_\Gamma$ and $\Pi_\Delta$, we have that $\Pi_\Gamma\vdash A$ and $\Pi_\Delta\vdash B$ are derivable in $\mathcal{S}$.

\vspace{2ex}
\noindent\textit{Case 2.} Our derivation ends with $\str\vdash$, and we assume that the connective $\str$ is introduced within $\Delta$ (the case when it is introduced within $\Gamma$ is treated analogously):
\[
\f{\Pi_{\Gamma_1,\Delta_1}\vdash C \quad D,\Pi_{\Gamma_2,\Delta_2}\vdash A\kon B}{\Pi_{\Gamma,\Delta}\vdash A\kon B}.
\]
By Lemma~\ref{Lemma:Main0} and the induction hypothesis, we have that $\Gamma_1\vdash I$, $\Delta_1\vdash C$, $\Gamma_2\vdash A$ and $D,\Delta_2\vdash B$ are derivable in $\mathcal{S}$, from which one easily concludes that for some $\Pi_\Gamma$ and $\Pi_\Delta$, we have that $\Pi_\Gamma\vdash A$ and $\Pi_\Delta\vdash B$ are derivable in $\mathcal{S}$.
\end{proof}

\begin{rem}
Proposition~\ref{Lemma:Main1} is an analogue of \cite[Proposition 7.6]{KML71} (see also \cite[Lemma 2]{L68}). Note that in the formulation and the proof of \cite[Proposition 7.6]{KML71} the authors assumed that $\Gamma,\Delta\vdash A\kon B$ is a proper sequent (defined in terms of  ``shapes'', which are analogues of formulae in our setting). This turns out to be redundant (see \cite[page 2]{V77}). However, the propriety condition is necessary in the following proposition, which is an analogue of \cite[Proposition 7.8]{KML71} (see also \cite[Lemma 3]{L68}).
\end{rem}

\begin{prop} \label{Lemma:Main2}
Let $\Pi_{\Gamma,A\str B,\Delta}\vdash C$ be a proper sequent derivable in $\mathcal{S}$, and let $\Gamma,A$ and $B,\Delta,C$ be disjoint. Then $\Gamma\vdash A$ and $B,\Delta\vdash C$ are derivable in $\mathcal{S}$.
\end{prop}
\begin{proof}
We proceed by induction on complexity of a cut-free derivation of the sequent $\Pi_{\Gamma,A\str B,\Delta}\vdash C$ to prove that if $\Gamma,A\str B,\Delta\vdash C$ is proper, while $\Gamma,A$ and $B,\Delta,C$ are disjoint, then for some $\Pi_\Gamma$ and $\Pi_{B,\Delta}$, the sequents $\Pi_{\Gamma}\vdash A$ and $\Pi_{B,\Delta}\vdash C$ are derivable in $\mathcal{S}$. By Remark~\ref{Remark:PiGamma}, this suffices for our proof.

The basis of this induction, when our derivation consists just of an axiom, is again straightforward. Also, the cases when our derivation ends with weakening, interchange, $\kon\vdash$ or $\vdash\str$ are easy to deal with. The propriety of the conclusion of these rules guarantees that the premises are proper too---hence, one may apply the induction hypothesis. It remains to consider the following two cases.

\vspace{2ex}
\noindent\textit{Case 1.} Our derivation ends with $\vdash\kon$. Let us assume that $A\str B$ occurs in the left premise (when $A\str B$ occurs in the right premise, we proceed in the same manner):
\begin{prooftree}
\AxiomC{$\Pi_{\Gamma_1,A\str B,\Delta_1}\vdash C_1$}
\AxiomC{$\Pi_{\Gamma_2,\Delta_2}\vdash C_2$}
\RightLabel{.}
\BinaryInfC{$\Pi_{\Gamma,A\str B,\Delta}\vdash C$}
\end{prooftree}
By the induction hypothesis applied to the left premise we have that $\Pi_{\Gamma_1}\vdash A$ and $\Pi_{B,\Delta_1}\vdash C_1$ are derivable in $\mathcal{S}$. By Lemma~\ref{Lemma:Main0} applied to the right premise ($\Gamma_2$ and $C_2$ are disjoint), we have that $\Gamma_2\vdash I$ and $\Delta_2\vdash C_2$ are derivable in $\mathcal{S}$. Applying $\vdash\kon$ to $\Pi_{B,\Delta_1}\vdash C_1$ and $\Delta_2\vdash C_2$ we have that $\Pi_{B,\Delta}\vdash C$, for $\Pi_{B,\Delta}=\Pi_{B,\Delta_1},\Delta_2$, is derivable in $\mathcal{S}$. From the following derivation
\begin{prooftree}
\AxiomC{$\Gamma_2\vdash I$}
\AxiomC{$\Pi_{\Gamma_1}\vdash A$}
\BinaryInfC{$\Gamma_2,\Pi_{\Gamma_1}\vdash I\kon A$}
\AxiomC{$A\vdash A$}
\UnaryInfC{$I,A\vdash A$}
\UnaryInfC{$I\kon A\vdash A$}
\BinaryInfC{$\Gamma_2,\Pi_{\Gamma_1}\vdash A$}
\end{prooftree}
it follows that $\Pi_{\Gamma}\vdash A$, for $\Pi_{\Gamma}=\Gamma_2,\Pi_{\Gamma_1}$, is derivable in $\mathcal{S}$.

\vspace{2ex}
\noindent\textit{Case 2.} Our derivation ends with $\str\vdash$. There are three essentially different subcases of this case.

\vspace{1ex}
\noindent\textit{Case 2.1.} The connective $\str$ introduced by this rule is the main connective in $A\str B$. Then our derivation has the following form:
\begin{prooftree}
\AxiomC{$\Pi_{\Gamma_1,\Delta_1}\vdash A$}
\AxiomC{$B,\Pi_{\Gamma_2,\Delta_2}\vdash C$}
\RightLabel{.}
\BinaryInfC{$\Pi_{\Gamma,A\str B,\Delta}\vdash C$}
\end{prooftree}
By Lemma~\ref{Lemma:Main0}, we have that $\Pi_{\Delta_1}\vdash I$, $\Pi_{\Gamma_1}\vdash A$, $\Pi_{\Gamma_2}\vdash I$ and $B,\Pi_{\Delta_2}\vdash C$ are derivable in $\mathcal{S}$. From this, one easily concludes (as in Case~1) that for some $\Pi_\Gamma$ and $\Pi_{B,\Delta}$, the sequents $\Pi_\Gamma\vdash A$ and $\Pi_{B,\Delta}\vdash C$ are derivable in $\mathcal{S}$.

\vspace{1ex}
\noindent\textit{Case 2.2.} The connective $\str$ is introduced within $\Gamma$ and $A\str B$ is in the right premise. We proceed similarly (without appealing to the fact that $\Gamma,A\str B,\Delta\vdash C$ is proper), when the connective $\str$ is introduced within $\Delta$ and $A\str B$ is either in the left premise or in the right premise. Hence, our derivation is of the following form:
\begin{prooftree}
\AxiomC{$\Pi_{\Gamma_1,\Delta_1}\vdash D$}
\AxiomC{$E,\Pi_{\Gamma_2,A\str B,\Delta_2}\vdash C$}
\RightLabel{.}
\BinaryInfC{$\Pi_{\Gamma,A\str B,\Delta}\vdash C$}
\end{prooftree}
By the induction hypothesis applied to the right premise we have that $\Pi_{E,\Gamma_2}\vdash A$ (hence $E,\Pi_{\Gamma_2}\vdash A$, by Remark~\ref{Remark:PiGamma}) and $\Pi_{B,\Delta_2}\vdash C$ are derivable in $\mathcal{S}$. By Lemma~\ref{Lemma:Main0} applied to the left premise we have that $\Pi_{\Delta_1}\vdash I$ and $\Pi_{\Gamma_1}\vdash D$ are derivable in $\mathcal{S}$. Applying $\str\vdash$ to $\Pi_{\Gamma_1}\vdash D$ and  $E,\Pi_{\Gamma_2}\vdash A$ we have that $\Pi_{\Gamma}\vdash A$, for $\Pi_{\Gamma}=\Pi_{\Gamma_1},D\str E,\Pi_{\Gamma_2}$, is derivable in $\mathcal{S}$. As in Case~1, since $\Pi_{\Delta_1}\vdash I$ and $\Pi_{B,\Delta_2}\vdash C$ are derivable, we conclude that $\Pi_{B,\Delta}\vdash C$, for $\Pi_{B,\Delta}=\Pi_{\Delta_1},\Pi_{B,\Delta_2}$, is derivable in $\mathcal{S}$.

\vspace{1ex}
\noindent\textit{Case 2.3.} The connective $\str$ is introduced within $\Gamma$ and $A\str B$ is in the left premise. Then our derivation has the following form:
\begin{prooftree}
\AxiomC{$\Pi_{\Gamma_1,A\str B,\Delta_1}\vdash D$}
\AxiomC{$E,\Pi_{\Gamma_2,\Delta_2}\vdash C$}
\RightLabel{.}
\BinaryInfC{$\Pi_{\Gamma,A\str B,\Delta}\vdash C$}
\end{prooftree}
Note that we cannot apply the induction hypothesis to the left premise because $\Gamma_1$ and $D$ need not be disjoint. However, we can apply Lemma~\ref{Lemma:Main0} to the right premise (the sequences $E,\Gamma_2$ and $\Delta_2,C$ are disjoint), and we obtain that $E,\Gamma_2\vdash I$, and $\Delta_2\vdash C$ are derivable in $\mathcal{S}$.

From the fact that $\Gamma,A\str B,\Delta\vdash C$ is proper it follows that $E,\Gamma_2\vdash I$, $\Gamma_1,A\str B,\Delta_1\vdash D$ and $D\str E$ are proper too. Hence, by Lemma \ref{Lemma:ConstProper} we conclude that $E,\Gamma_2$ is constant. In particular, formula $E$ is constant, so $D$ is also constant (because $D\str E$ is proper). Using Lemma \ref{Lemma:ConstProper} again, we have that $\Gamma_1,A\str B,\Delta_1$ is constant. Together with the fact that $\Gamma_2$ is constant, this asserts that $\Gamma$ and $A$ are constant, and by Corollary \ref{Corollary:ConstConst} we conclude that $\Gamma\vdash A$ is derivable in $\mathcal{S}$.

Since $B$ and $\Delta_1$ are also constant, by the same corollary, we have that $B,\Delta_1\vdash I$ is derivable in $\mathcal{S}$. Together with the fact that $\Delta_2\vdash C$ is derivable, this entails (as in Case~1) that $\Pi_{B,\Delta}\vdash C$, for $\Pi_{B,\Delta}=B,\Delta_1,\Delta_2$, is derivable in $\mathcal{S}$.
\end{proof}

\begin{rem}
Note that Case 2.3 is the only place in the proof of Proposition~\ref{Lemma:Main2} where we use the condition that the sequent $\Gamma,A\str B,\Delta\vdash C$ is proper. However, this condition cannot be omitted. To make sure of that, consider the following derivation of the sequent $(p\str I)\str I, p\str I \vdash I$.
\begin{prooftree}
\AxiomC{$p\str I\vdash p\str I$}
\AxiomC{$I\vdash I$}
\BinaryInfC{$p\str I, (p\str I)\str I \vdash I$}
\UnaryInfC{$(p\str I)\str I, p\str I \vdash I$}
\end{prooftree}
When we take that $\Gamma=(p\str I)\str I$, $A=p$, $B=I$, $\Delta$ is empty and $C=I$, it is evident that $\Gamma,A$ and $B,\Delta,C$ are disjoint. However, it is not hard to show that $(p\str I)\str I\vdash p$ is not derivable in $\mathcal{S}$.

Moreover, by a careful examination of Case 2.3, it can be noticed that we do not use propriety of the whole sequent $\Gamma,A\str B,\Delta\vdash C$, but only of its antecedent. However, propriety of the whole sequent $\Gamma,A\str B,\Delta\vdash C$ is necessary in order to be able to apply the induction hypothesis in the case when our derivation ends with $\vdash\str$.
\end{rem}

\section{The system \texorpdfstring{$\mathcal{IL}$}{IL}}\label{IL}

In this section we modify our system $\mathcal{S}$ into the system $\mathcal{IL}$ having single premise-single conclusion sequents. The sequents of the system are of the form $G\vdash A$, where $G$ and $A$ are formulae in which $\otm$ is \emph{strictly} associative and $I$ is the \emph{strict} neutral. We omit the brackets of the form $A\kon(B\kon C)$ and $(A\kon B)\kon C$ and write simply $A\kon B\kon C$. Also, we write $A$ instead of $A\kon I$ and $I\kon A$. The outermost brackets are also omitted. We call such expressions $\alpha$-formulae.

The axiomatic sequents are of the form $A\vdash A$, and there are two structural inference figures.
\[
\f{G\kon A\kon B\kon E\vdash D}{G\kon B\kon A\kon E\vdash
D}\pravilo{\rm interchange} \mbox{\hspace{9em}} \f{C\vdash
A\quad G\kon A\kon E\vdash D}{G\kon C\kon E\vdash
D}\pravilo{cut}
\]
The operational inference figures are the following.
\[
\f{C\vdash A\quad B\kon E\vdash D}{C\kon (A\str B)\kon E
\vdash D}\pravilo{$\str\vdash$}\mbox{\hspace{3.5em}}
\f{A\kon G\vdash D}{G\vdash A\str D}\pravilo{$\vdash\str$}\mbox{\hspace{3.5em}}\f{A\vdash C\quad B\vdash D}{A\kon B\vdash C\kon D}\pravilo{$\kon\vdash\kon$}\mbox{\hspace{3.5em}}
\]
The symbol $\kon$ in these figures should be envisaged as an operation on $\alpha$-formulae, rather than a connective. Of course, this is an abuse of the notation. With $\kon$ on the left-hand side being the operation and on the right-hand side being the connective, we have the following convention.
\[
A\kon B=\left\{ \begin{array}{rl}
A\kon B, & \mbox{\rm{if both }} A \mbox{\rm{ and }} B \mbox{\rm{ are not }} I;
\\
B, & \mbox{\rm{if }} A \mbox{\rm{ is }} I;
\\
A, & \mbox{\rm{if }} B \mbox{\rm{ is }} I.
\end{array}\right .
\]
With this in mind, the rule $\vdash\str$ has the following instance.
\[
\f{A\vdash C}{I\vdash A\str C}
\]

For a standard formula $A$, one obtains the $\alpha$-formula $A^\ast$ by deleting superfluous brackets and $I$'s. In the same way, for a sequence $\Gamma=A_1,\ldots,A_n$ of formulae, we have $\Gamma^\ast=(\ldots(A_1\kon A_2)\kon \ldots \kon A_n)^\ast$. When $\Gamma$ is empty, $\Gamma^\ast$ is $I$.

The following proposition, whose proof is given in Appendix~\ref{sec:appendix}, justifies how $\mathcal{S}$ and $\mathcal{IL}$ are close to each other.

\begin{prop}\label{translation}
A sequent $\Gamma\vdash B$ has a (cut-free) derivation in $\mathcal{S}$ if and only if $\Gamma^\ast\vdash B^\ast$ has a (cut-free) derivation in $\mathcal{IL}$.
\end{prop}

A language for coding derivations in $\mathcal{IL}$ is based on terms with types. In the expression $f\colon A\vdash B$, we say that $f$ is a \emph{term} whose \emph{type} is $A\vdash B$, where $A$ and $B$ are $\alpha$-formulae. The \emph{primitive} terms are: $\mj_A\colon A\vdash A$, $\mc_{B,A}\colon B\kon A\vdash A\kon B$, $\eta_{A,B}\colon B\vdash A\str(A\kon B)$ and $\varepsilon_{A,B}\colon A\kon(A\str B)\vdash B$. The terms are built according to the following rules:

\vspace{1ex}

(1) if $f\colon  A\vdash B$ and $g\colon  B\vdash C$ are terms, then
$(g\circ f)\colon  A\vdash C$ is a term;

\vspace{1ex}

(2) if $f_1\colon  A_1\vdash B_1$ and $f_2\colon  A_2\vdash B_2$ are terms,
then $(f_1\kon f_2)\colon  A_1\kon A_2\vdash B_1\kon B_2$ is a term;

\vspace{1ex}

(3) if $f\colon  B_1\vdash B_2$ is a term and $A$ is an $\alpha$-formula, then
$(A\str f)\colon  A\str B_1\vdash A\str B_2$ is a term.

\begin{rem}\label{language1}
We assume that the language of terms is \emph{strict} in the similar sense as the language of $\alpha$-formulae---we identify the terms $f\kon(g\kon h)$, $(f\kon g)\kon h$ with $f\kon g\kon h$, and $f\kon\mj_I$, $\mj_I\kon f$ with $f$. Also, we omit the outermost brackets in terms.

With our convention for $\alpha$-formulae, $\mc_{I,I}$ has the type $I\vdash I$. Moreover, $\eta_{I,I}\circ\varepsilon_{I,I}\colon I\str I\vdash I\str I$ and $\varepsilon_{I,I}\circ\eta_{I,I}\colon I\vdash I$ are well defined terms.
\end{rem}

The terms are taken modulo congruence generated by the following equalities:
\begin{equation}\label{cat1}
f\circ\mj_A=f\quad\mbox{\rm and}\quad \mj_B\circ f=f,\quad \mbox{\rm for}\; f\colon A\vdash B,
\end{equation}
\begin{equation}\label{cat2}
h\circ(g\circ f)=(h\circ g)\circ f,
\end{equation}
\begin{equation}\label{fun1}
\mj_A\kon \mj_B=\mj_{A\kon B},
\end{equation}
\begin{equation}\label{fun2}
(g_1\kon g_2)\circ(f_1\kon f_2)=(g_1\circ f_1)\kon(g_2\circ f_2),
\end{equation}
\begin{equation}\label{nat}
\mc_{A',B'}\circ(f\kon g)=(g\kon f)\circ \mc_{A,B},\quad \mbox{\rm for}\; f\colon A\vdash A'\; \mbox{\rm and}\; g\colon B\vdash B',
\end{equation}
\begin{equation}\label{iso}
\mc_{B,A}\circ \mc_{A,B}=\mj_{A\kon B},
\end{equation}
\begin{equation}\label{coh}
\mc_{A\kon B,C}=(\mc_{A,C}\kon \mj_B)\circ(\mj_A\kon
\mc_{B,C}),
\end{equation}
\begin{equation}\label{fun2str}
A\str(g\circ f)=(A\str g)\circ (A\str f),
\end{equation}
\begin{equation}\label{nateta}
\eta_{A,B'}\circ f=(A\str(\mj_A\kon f))\circ\eta_{A,B}, \quad \mbox{\rm for}\; f\colon B\vdash B',
\end{equation}
\begin{equation}\label{fun1str}
A\str\mj_B=\mj_{A\str B},
\end{equation}
\begin{equation}\label{nateps}
\varepsilon_{A,B'}\circ(\mj_A\kon(A\str
f))=f\circ\varepsilon_{A,B}, \quad \mbox{\rm for}\; f\colon B\vdash B',
\end{equation}
\begin{equation}\label{triang1}
\varepsilon_{A,A\kon B}\circ (\mj_A\kon \eta_{A,B})=\mj_{A\kon B},
\end{equation}
\begin{equation}\label{triang2}
(A\str\varepsilon_{A,B})\circ \eta_{A,A\str B}=\mj_{A\str B}.
\end{equation}
Note that only terms with the same type could be equal. In the proof below we will not mention very frequent application of the equality \ref{cat2} and we will omit the brackets of the form $(h\circ g)\circ f$ and $h\circ(g\circ f)$.

Every derivation of $A\vdash B$ in $\mathcal{IL}$ could be coded by a term whose type is $A\vdash B$. In the coding below, we omit the indices when they are clear from contexts and also we write $A\kon$, $\kon A$ instead of $\mj_A\kon$, $\kon\mj_A$. The axiomatic sequent $A\vdash A$ is coded by $\mj_A\colon A\vdash A$. We assume that the derivations of premisses in the inference figures are already coded by terms $f$ or $g$ and then we obtain the following:
\[
\f{f\colon G\kon A\kon B\kon E\vdash D}{f\circ (G\kon
\mc\kon E)\colon G\kon B\kon A\kon E\vdash
D}\pravilo{interchange}
\]
\[
\f{f\colon C\vdash A\quad g\colon G\kon A\kon E\vdash D}{g\circ
(G\kon f\kon E)\colon G\kon C\kon E\vdash D}\pravilo{cut}
\]
\[
\f{f\colon C\vdash A\quad g\colon B\kon E\vdash
D}{g\circ(\varepsilon\kon E)\circ(f\kon(A\str
B)\kon E)\colon C\kon (A\str B)\kon E\vdash
D}\pravilo{$\str\vdash$}
\]
\[
\f{f\colon A\kon G\vdash D}{(A\str f)\circ \eta\colon G\vdash A\str
D}\pravilo{$\vdash\str$}
\]
\[
\f{f\colon A\vdash C\quad  g\colon B\vdash D}{f\kon g\colon A\kon B\vdash C\kon D}\pravilo{$\kon\vdash\kon$}
\]

\begin{prop}\label{coding}
For every term $f\colon A\vdash B$ there is a derivation of $A\vdash B$ in $\mathcal{IL}$ coded by a term equal to $f$.
\end{prop}
\begin{proof}
We proceed by induction on complexity of $f$. Let us consider the primitive term $\eta_{A,B}\colon B\vdash A\str(A\kon B)$. Then the following derivation in $\mathcal{IL}$
\[
\f{A\kon B\vdash A\kon B}{B\vdash A\str
(A\kon B)}
\]
is coded by the term $(A\str \mj_{A\otimes B})\circ\eta_{A,B}$.
From \ref{fun1str} and \ref{cat1} it follows that $\eta_{A,B}=(A\str \mj_{A\otimes B})\circ\eta_{A,B}$. We proceed in the same manner for other primitive terms.

In the induction step the cases corresponding to (1) and (2) in the inductive definition of terms are trivial and we consider just the case when $f\colon A\str B_1\vdash A\str B_2$ is of the form $A\str g$, for $g\colon B_1\vdash B_2$. By the induction hypothesis, there exists a term $g'$  equal to $g$, which is the code for a derivation of $B_1\vdash B_2$. Then the following derivation in $\mathcal{IL}$
\[
\f{\f{\mj_A\colon A\vdash A\quad g'\colon B_1\vdash B_2}{A\kon(A\str B_1)\vdash B_2}}{A\str B_1\vdash A\str B_2}
\]
is coded by the term
\[
(A\str(g'\circ(\varepsilon_{A,B_1}\circ(\mj_A\kon(A\str B_1)))))\circ \eta_{A,A\str B_1}.
\]
This term is by \ref{fun1} and \ref{cat1} equal to $(A\str(g'\circ\varepsilon_{A,B_1}))\circ \eta_{A,A\str B_1}$, which is by \ref{fun2str}, \ref{triang2} and \ref{cat1} equal to $A\str g'$, and hence to $A\str g$.
\end{proof}

\begin{definition}
An $\alpha$-formula is called \emph{prime} when it is not of the form $A\kon B$. If $A_1,\ldots,A_n$ are prime, then they are the \emph{prime factors} of $A_1\kon\ldots\kon A_n$.
\end{definition}

A proof of the following result is left for Appendix~\ref{sec:appendix}.

\begin{thm}[Cut-elimination strengthened]\label{cutelim1}
If $f\colon A\vdash B$ is a code of a derivation in $\mathcal{IL}$, then there is a cut-free derivation in $\mathcal{IL}$ of the sequent $A\vdash B$ coded by a term equal to~$f$.
\end{thm}

From Proposition~\ref{coding} and Theorem~\ref{cutelim1} one can deduce the following.
\begin{cor}\label{cutelim2}
For every term $f$, there is a cut-free derivation in $\mathcal{IL}$ coded by a term equal to $f$.
\end{cor}

\begin{definition}\label{assorted1}
A term is called \emph{central} when it is free of  $\eta$, $\varepsilon$ and the operation $A\str$. (For example, $\mc_{(p\str q)\kon q,p}\circ (\mj_{p\str q}\kon \mc_{p,q})$ is central.) An $\alpha$-formula is called \emph{constant} when it does not contain letters. An $\alpha$-formula is called \emph{assorted} when all its non-constant prime factors are mutually distinct. (For example, $(I\str I)\kon((p\kon q)\str r)\kon ((q\kon p)\str r)\kon(I\str I)$ is assorted.)
\end{definition}

\begin{rem}
For a central term $f\colon A\vdash B$ it holds that if one of $A$ or $B$ has no constant prime factors, or is equal to $I$ or is assorted, then the same holds for the other. In the first case, we say that $f$ is a \emph{non-constant central} term. In the second case, we say that $f$ is an $I$-\emph{central} term and in the last case we say that it is an \emph{assorted central} term. If a term is non-constant or $I$-central, then we call it \emph{reduced} central term.
\end{rem}

\begin{definition}\label{invertible}
A term $f\colon A\vdash B$ is \emph{invertible}, when there exists a term $g\colon B\vdash A$ such that $g\circ f=\mj_A$ and $f\circ g=\mj_B$. In this case, we call $g$ the \emph{inverse} of $f$ and denote it by $f^{-1}$.
\end{definition}

\begin{rem}\label{ItoI}
Note that $\varepsilon_{I,A}\colon I\str A\vdash A$ and $\eta_{I,A}\colon A\vdash I\str A$ are inverse to each other. We have that $\varepsilon_{I,A}\circ \eta_{I,A}=\mj_A$ is just an instance of \ref{triang1}, and for $\eta_{I,A}\circ \varepsilon_{I,A}=\mj_{I\str A}$ we rely on \ref{nateta} and then on \ref{triang2}. In particular, $\varepsilon_{I,I}\colon I\str I\vdash I$ and $\eta_{I,I}\colon I\vdash I\str I$ are inverse to each other.
\end{rem}

\begin{lem}\label{AtoI}
For every constant $\alpha$-formula $A$ there exists an invertible term $f\colon A\vdash I$.
\end{lem}
\begin{proof}
We proceed by induction on the number of occurrences of $\kon$ and $\str$ in $A$. If this number is 0, then $A$ is $I$ and the invertible term is $\mj_I$.

If $A$ is of the form $A_1\kon A_2$, then by the induction hypothesis we have invertible terms $f_1\colon A_1\vdash I$ and $f_2\colon A_2\vdash I$. By relying on the equalities \ref{fun2} and \ref{fun1}, it is easy to verify that $f_1\kon f_2\colon A\vdash I$ is invertible.

If $A$ is of the form $A_1\str A_2$, then for invertible terms $f_1\colon A_1\vdash I$ and $f_2\colon A_2\vdash I$, which exist by the induction hypothesis, consider the term
\[
\varepsilon_{I,I}\circ (I\str(f_2\circ \varepsilon_{A_1,A_2}\circ (f_1^{-1}\kon\mj_{A_1\str A_2})))\circ \eta_{I,A_1\str A_2}\colon A_1\str A_2\vdash I.
\]
With the help of Remark~\ref{ItoI}, one can prove that
\[
(A_1\str(\varepsilon_{I,A_2}\circ (f_1\kon(I\str A_2))))\circ \eta_{A_1,I\str A_2}\circ (I\str f_2^{-1})\circ\eta_{I,I}\colon I\vdash A_1\str A_2
\]
is its inverse.
\end{proof}

\begin{rem}\label{developed}
By relying on \ref{fun2}, every central term is equal to a term of the form $f_n\circ\ldots\circ f_1$ where each $f_i$ is built out of $\mj$ and $\mc$ with the help of~$\kon$. Moreover, by relying again on \ref{fun2} and on \ref{coh} when needed, one may assume that each $f_i$ contains only one occurrence of $\mc_{C,D}$, and $C$, $D$ are prime.
\end{rem}

\begin{lem}\label{balance}
For every central term $f$ there exists a reduced central term $f'$ and two invertible terms $u$ and $v$ such that $f=v^{-1}\circ f'\circ u$. Moreover, $u$ and $v$ depend just on the type of $f$, and if $f$ is assorted, then $f'$ is assorted~too.
\end{lem}
\begin{proof}
Let $f\colon A\vdash B$ be a central term and let $B_1,\ldots, B_m$ be the prime factors of $B=B_1\kon\ldots\kon B_m$. We define $v$ to be the term $v_1\kon\ldots\kon v_m$, where $v_i\colon B_i\vdash I$ is an invertible term, which exists according to Lemma~\ref{AtoI} when $B_i$ is constant, or otherwise $v_i$ is $\mj_{B_i}$. Note that $v$ is invertible and it depends only on $B$.

By Remark~\ref{developed} we may assume that $f$ is in a ``developed'' form $f_n\circ\ldots\circ f_1$, and by relying on \ref{fun2} and \ref{nat}, the term $v\circ f$ is equal to a term of the form $f'\circ u$. Here, $f'$ is a central term whose type is either $I\vdash I$, when $A$ and $B$ are constant, or $f'\colon A'\vdash B'$ is a non-constant central term. Moreover, for $A_1,\ldots,A_m$ being the prime factors of $A=A_1\kon\ldots\kon A_m$, the term $u$ is of the form $u_1\kon\ldots\kon u_m$ where $u_i\colon A_i\vdash I$ is an invertible term, which exists according to Lemma~\ref{AtoI} when $A_i$ is constant, or otherwise $u_i$ is $\mj_{A_i}$. Note that $u$ is invertible and it depends only on $A$. The prime factors of $A'$ and $B'$ are the non-constant prime factors of $A$ and $B$, hence if $f$ is assorted, $f'$ is assorted too. A formal proof of these facts proceeds by induction on the number $n-1\geq 0$ of occurrences of $\circ$ in $f$.
\end{proof}

\begin{rem}\label{cII}
Note that for every $C$ we have that $\mc_{I,C}=\mj_C=\mc_{C,I}$, which follows from the instance $\mc_{I,C}=\mc_{I,C}\circ \mc_{I,C}$ of \ref{coh} ($A=B=I$) with the help of \ref{iso}.
\end{rem}

\begin{prop}\label{centralcoh}
Two assorted and reduced central terms of the same type are equal.
\end{prop}
\begin{proof}
If $f\colon I\vdash I$ is an $I$-central term, then this term is built out of $\mc_{I,I}$ with the help of $\circ$ and $\kon$. By Remark~\ref{cII}, $\mc_{I,I}=\mj_I$, and then $f=\mj_I$.

If $f\colon A\vdash B$ is a non-constant term, then by using Remark~\ref{cII} we can get rid of subterms of the form $\mc_{C,I}$ and $\mc_{I,C}$ in $f$. Moreover, according to Remark~\ref{developed}, we may assume that $f$ is either $\mj_A$ or of the form $f_n\circ\ldots\circ f_1$ where each $f_i$ is built out of $\mj$, $\kon$ and just one $\mc_{C,D}$, with $C$, $D$ prime, non-constant. This means that $f_i$ is of the form
\[
\mj_{A_{i_1}\kon\ldots\kon A_{i_{j-1}}}\kon\mc_{A_{i_j},A_{i_{j+1}}} \kon \mj_{A_{i_{j+2}}\kon\ldots\kon A_{i_m}}.
\]
Therefore, $f_i$ corresponds to an adjacent transposition $\sigma_j$ in the symmetric group $S_m$, where $m$ is the number of prime factors in $A$ (and $B$).

The standard presentation of the symmetric group $S_m$ in terms of adjacent transpositions $\sigma_i=(i,i+1)$ is given by the following three equalities:
\[
\sigma_i^2=1,\quad \sigma_i\sigma_j=\sigma_j\sigma_i,\; \mbox{\rm when } |i-j|>1,\quad (\sigma_i\sigma_{i+1})^3=1.
\]
These equalities correspond to \ref{iso}, some instances of \ref{fun2} and the equality
\[
(\mc_{B,C}\kon\mj_A)\circ(\mj_B\kon\mc_{A,C})\circ (\mc_{A,B}\kon\mj_C)= (\mj_C\kon\mc_{A,B})\circ(\mc_{A,C}\kon\mj_B) \circ (\mj_A\kon\mc_{B,C}),
\]
which is derivable with two applications of \ref{coh} and one of \ref{nat}. This means that by relying on the standard presentation of symmetric groups (see \cite{M896} and \cite[Note C, pp.\ 464-465]{B11}) one may conclude that two assorted and reduced central terms of the same type, since they could be presented by adjacent transpositions and they correspond to the same permutation, are equal. For the sake of completeness, we give an outline of the proof of this classical result. (An alternative proof is given in \cite[Section~5.2]{DP04}.)

If for $i\geq j$, one abbreviates the element $\sigma_i\sigma_{i-1}\ldots\sigma_{j+1}\sigma_j$ of $S_m$ by $\sigma_{[i,j]}$, the equalities listed above are sufficient to present every element of $S_m$ in a \emph{normal form} 1, or for $n\geq 1$
\[
\sigma_{[i_1,j_1]}\ldots\sigma_{[i_n,j_n]},
\]
with $i_1<i_2<\ldots<i_n$. This normal form is implicit in \cite[Note C, pp.\ 464-465]{B11}. Therefore, we have at least $m!$ such normal forms, and by induction on $m$ it is easy to show that there are exactly $m!$ of them. This means that two non-identical normal forms correspond to different permutations.

Take now two assorted and reduced central terms of the same type. Our equalities are sufficient to ``put them in normal forms''. These normal forms must be identical, otherwise, the terms could not be of the same type.
\end{proof}

\begin{cor} \label{Corollary:CentralTerms}
Two assorted central terms of the same type are equal.
\end{cor}
\begin{proof}
Let $f,g\colon A\vdash B$ be two assorted central terms. From Lemma~\ref{balance} there exist two assorted and reduced central terms $f',g'\colon A'\vdash B'$ such that $f=v\circ f'\circ u$ and $g=v\circ g'\circ u$. Since $u$ and $v$ are invertible, we have that $f=g$ if and only if $f'=g'$. It remains to apply Proposition~\ref{centralcoh}.
\end{proof}

\begin{cor}\label{II}
Every term $f\colon I\vdash I$ is equal to $\mj_I$.
\end{cor}
\begin{proof}
Take a cut-free derivation of $I\vdash I$ coded by a term $f'$ equal to $f$. Then $f'$ must be central and by Proposition~\ref{centralcoh} it is equal to $\mj_I$.
\end{proof}

\begin{definition}
A cut-free derivation in $\mathcal{IL}$ is called \emph{clean} when it does not contain applications of the rule $\kon\vdash\kon$ with an upper sequent of the form $I\vdash I$ (we call such an application \emph{redundant} $\kon\vdash\kon$), and an application of interchange which permutes $I$ with some other $\alpha$-formula (we call such an application \emph{invisible} interchange).
\end{definition}

\begin{prop}\label{clean}
For every term $f$ there is a clean derivation in $\mathcal{IL}$ coded by a term equal to $f$.
\end{prop}
\begin{proof}
Let $\mathcal{D}$ be a cut-free derivation coded by a term equal to $f$. Replace every part of $\mathcal{D}$ ending with a redundant $\kon\vdash\kon$ with the derivation of the other upper sequent of this rule. By Corollary~\ref{II}, the obtained derivation is coded by a term equal to $f$. Then omit every invisible interchange in the obtained derivation. By Remark~\ref{cII} the new derivation is again coded by a term equal to $f$.
\end{proof}

\section{The main result}\label{main section}

This section is devoted to a proof of Theorem~\ref{main}, which is the main result of this paper. We start with rephrasing two definitions.

\begin{definition}
An $\alpha$-formula $A$ is \emph{proper} when every ordinary formula $B$ such that $B^\ast=A$ is proper. A sequent $A\vdash B$ of $\mathcal{IL}$ is \emph{proper} when both $A$ and $B$ are proper.
\end{definition}

\begin{definition}
Two sequences of $\alpha$-formulae are \emph{disjoint} when
there is no propositional letter occurring simultaneously in
a formula from one and a formula from the other.
For an $\alpha$-formula $A$, let $\Pi_A$ be obtained by permuting the prime factors in~$A$.
\end{definition}

In order to prove our main result, we need to put Lemma \ref{Lemma:ConstProper} and Propositions \ref{Lemma:Main1} and \ref{Lemma:Main2} in the context of the system $\mathcal{IL}$.  For this, we rely on Proposition~\ref{translation} together with Proposition~\ref{coding} and Corollary~\ref{cutelim2}.

\begin{lem} \label{Lemma:ConstProperIL}
Suppose that $f\colon A\vdash B$ is a term, where $A$ is proper and $B$ is constant. Then $A$ is constant.
\end{lem}

\begin{prop} \label{Lemma:Main1IL}
Let $f\colon \Pi_{A\kon B}\vdash C\kon D$ be a term, where $A,C$ and $B,D$ are disjoint. Then there exist terms $g\colon A\vdash C$ and $h\colon B\vdash D$.
\end{prop}

\begin{prop} \label{Lemma:Main2IL}
Let $f\colon \Pi_{A\kon (B\str C)\kon D}\vdash E$ be a term whose type is proper, where $A,B$ and $C,D,E$ are disjoint. Then there exist terms $g\colon A\vdash B$ and $h\colon C\kon D\vdash E$.
\end{prop}

Also, we will need the following auxiliary results.

\begin{lem} \label{Lemma:EvenNumber}
Let $f\colon A\vdash B$ be a term. Then every propositional letter has an even number of occurrences in $A\vdash B$.
\end{lem}
\begin{proof}
We proceed by induction on complexity of $f$. If the complexity of $f$ is $0$, then $f$ is a primitive term and the lemma obviously holds. The cases when $f$ is of the form $f_1\kon f_2$ or $A\str g$ are trivial.

Suppose now that $f$ is of the form $f_2\circ f_1$ for $f_1\colon A\vdash C$ and $f_2\colon C\vdash B$. Suppose that $p$ occurs $a$ times in $A$, $b$ times in $B$ and $c$ times in $C$. From the induction hypothesis we have that $a+c$ and $b+c$ are even. Since $a+b=(a+c)+(b+c)-2c$, we conclude that $a+b$ (the number of occurrences of $p$ in $A\vdash B$) is also even.
\end{proof}

\begin{prop}\label{R1}
If $A$ is a proper and non-constant $\alpha$-formula, then there exists an $I$-free $\alpha$-formula $A'$ and an invertible term $f\colon A\vdash A'$.
\end{prop}
\begin{proof}
We proceed by induction on the number $n\geq 0$ of occurrences of $\kon$ and $\str$ in $A$. If $n=0$, then $A$ must be a propositional letter $p$ and $f$ is $\mj_p$.

Assume that $A$ is of the form $A_1\kon A_2$. If both $A_1$ and $A_2$ are non-constant, then by the induction hypothesis there are $I$-free $A_1'$ and $A_2'$ as well as invertible terms $f_1\colon A_1\vdash A_1'$ and $f_2\colon A_2\vdash A_2'$. We take $A'$ to be $A_1'\kon A_2'$ and $f$ to be $f_1\kon f_2$.

If $A_1$ is constant, then $A_2$ must be non-constant, and vice versa. We apply Lemma~\ref{AtoI} together with the induction hypothesis in order to obtain invertible terms $f_1\colon A_1\vdash I$ and $f_2\colon A_2\vdash A_2'$, where $A_2'$ is $I$-free. In this case $A'$ is $A_2'$ and $f$ is $f_1\kon f_2$.

\vspace{1ex}
Finally, let $A$ be of the form $A_1\str A_2$. Since $A$ is proper and non-constant it is not possible for $A_2$ to be constant. If both $A_1$ and $A_2$ are non-constant, then by the induction hypothesis there are $I$-free $A_1'$ and $A_2'$ as well as invertible terms $f_1\colon A_1\vdash A_1'$ and $f_2\colon A_2\vdash A_2'$. We take $A'$ to be $A_1'\str A_2'$ and $f$ to be
\[
(A_1'\str(f_2\circ \varepsilon_{A_1,A_2}\circ (f_1^{-1}\kon\mj_{A_1\str A_2})))\circ \eta_{A_1',A_1\str A_2}\colon A_1\str A_2\vdash A_1'\str A_2'.
\]
With the help of equalities \ref{fun2str}-\ref{triang2}, one concludes that $f$ is invertible and its inverse is
\[
(A_1\str(\varepsilon_{A_1',A_2}\circ (f_1\kon(A_1'\str A_2))))\circ \eta_{A_1,A_1'\str A_2}\circ (A_1'\str f_2^{-1}).
\]

The case with $A_1$ constant and $A_2$ non-constant is analogous---let $f_1\colon A_1\vdash I$ be an invertible term that exists by Lemma~\ref{AtoI} and let $f_2\colon A_2\vdash A_2'$, where $A_2'$ is $I$-free, be an invertible term that exists by the induction hypothesis. An invertible $f'\colon A_1\str A_2\vdash I\str A_2'$ is constructed as $f$ in the preceding paragraph. We take now $f$ to be $\varepsilon_{I,A_2'}\circ f'\colon A_1\str A_2\vdash A_2'$ and it is invertible since $\varepsilon_{I,A_2'}$ is invertible by Remark~\ref{ItoI}. (This case is a non-essential generalization of the last case in the proof of Lemma~\ref{AtoI}.)
\end{proof}

\begin{prop}\label{R2}
There is no derivation in $\mathcal{IL}$ of $A\vdash I$, with $A$ an $I$-free $\alpha$-formula.
\end{prop}
\begin{proof}
Otherwise, there will be a clean derivation of $A\vdash I$, which is impossible. Indeed, $p\vdash I$ has no clean derivation, and if $A\vdash I$ is a sequent of lowest complexity derived by a clean derivation, then one immediately obtains a clean derivation of a sequent of lower complexity.
\end{proof}

\begin{definition}
We say that an occurrence of a propositional letter in a formula is \textit{positive} (\textit{negative}) if it occurs in an even (odd) number of antecedents of implications. The sign of occurrence in the sequent $A\vdash B$ is the same as in the formula $A\str B$. We say that the sequent is \textit{balanced} if each letter occurring in it has exactly two occurrences and they have opposite
signs. The term $f\colon A\vdash B$ is \textit{balanced} if the sequent $A\vdash B$ is balanced.
\end{definition}

\begin{rem}\label{balanced1}
From Lemma~\ref{Lemma:EvenNumber} it easily follows that every sequent in a cut-free derivation of a balanced sequent is balanced. Hence, for $A$ being $I$-free, there is no derivation of a balanced sequent $A\vdash A$ ending with $\str\vdash$.
\end{rem}

\begin{rem}\label{R3}
By Proposition~\ref{R1}, Lemma~\ref{AtoI} and Proposition~\ref{R2}, for a proper derivable sequent $A\vdash B$, there is a sequent $A'\vdash B'$ with either both $A'$, $B'$ being $I$-free, or $A'$ being $I$ and $B'$ being $I$-free, or both being $I$, and a pair of invertible terms $u\colon A\vdash A'$ and $v\colon B\vdash B'$. Moreover, it is easy to see that if $A\vdash B$ is balanced, then $A'\vdash B'$ is balanced too. For a pair of terms $f,g\colon A\vdash B$, by invertibility of $u$ and $v$, we have that $f=g$ if and only if $v\circ f\circ u^{-1}=v\circ g\circ u^{-1}$.
\end{rem}

\begin{definition} \label{Definition:Complexity}
The \textit{complexity} of a sequent is the number of occurrences of $\kon$ plus double the number of occurrences of $\str$ in it. Let the \emph{complexity} of a term be the number of symbols in it.
\end{definition}

\begin{thm}\label{main}
Two balanced terms of the same proper type are equal.
\end{thm}
\begin{proof}
There are a lot of cases to discuss in this proof and for the sake of brevity, we will not refer in every calculation to equalities from the list given in Section~\ref{IL}, after Remark~\ref{language1}. Moreover, by appealing to Proposition~\ref{clean}, we assume that every derivation whose existence is guaranteed by Propositions \ref{Lemma:Main1IL} and \ref{Lemma:Main2IL} is clean.

Let $A\vdash B$ be the proper type of our balanced terms. By Remark~\ref{R3}, we may assume that either both $A$ and $B$ are $I$-free, or $A$ is $I$ and $B$ is $I$-free, or they are both $I$. The last case is solved by Corollary~\ref{II}.

For $A\vdash B$ as above, let $f\colon A\vdash B$ and $g\colon\Pi_A\vdash B$ be two balanced terms. By Proposition~\ref{clean}, one may assume that there are clean derivations $\mathcal{D}_f$ of $A\vdash B$ and $\mathcal{D}_g$ of $\Pi_A\vdash B$ in $\mathcal{IL}$ coded by $f$ and $g$ respectively. By our assumption it follows that every sequent in $\mathcal{D}_f$ and $\mathcal{D}_g$ is either $I$-free or its antecedent is $I$ and its consequent is $I$-free.

Let $n$ be the complexity of the sequent $A\vdash B$, and let $m$ be the sum of complexities of the terms $f$ and $g$. We proceed by induction on lexicographically ordered \emph{complexity pairs} $(n,m)$ in order to prove that there is a central term $\sigma\colon A\vdash \Pi_A$ such that $f=g\circ\sigma$. Note that since $A\vdash B$ is balanced, we have that $A$ is assorted (see Definition~\ref{assorted1}). This suffices for our proof since, when $\Pi_A=A$, then from Corollary~\ref{Corollary:CentralTerms} it follows that $\sigma\colon A\vdash A$ is equal to $\mj_A\colon A\vdash A$.

The basis of this induction is when both $f$ and $g$ are of the form $\mj_p\colon p\vdash p$, and then for $\sigma=\mj_p$ we have $f=g\circ\sigma$. For the induction step, the case when both $\mathcal{D}_f$ and $\mathcal{D}_g$ are axioms is trivial. Moreover, the case when one of $\mathcal{D}_f$ or $\mathcal{D}_g$ ends with interchange, is solved just by appealing to the induction hypothesis applied to the pair of derivations in which one is shortened by the interchange rule (the first component of the complexity pair remains the same and the second decreases).

It remains, aside from arguments by symmetry, to consider the following cases depending on the last inference rules in $\mathcal{D}_f$ and $\mathcal{D}_g$. The reason for lacking the cases with $\mathcal{D}_f$ being an axiom and $\str\vdash$ being the last rule in $\mathcal{D}_g$ is lying in Remark~\ref{balanced1} (if there is no derivation of $A\vdash A$ ending with $\str\vdash$, then there is no derivation of $\Pi_A\vdash A$ ending with this rule). Also, by our assumptions on the sequent $A\vdash B$, the case when one derivation ends with $\kon\vdash \kon$ and the other with $\vdash\str$ is impossible.

\begin{center}
\begin{tabular}{|c|c|c|}
\hline
Case & The last rule in $\mathcal{D}_f$ & The last rule in $\mathcal{D}_g$ \\
\hline
1 & axiom, main connective $\kon$ & $\kon\vdash\kon$ \\
\hline
2 & axiom, main connective $\str$ & $\vdash\str$ \\
\hline
3 & $\kon\vdash\kon$ & $\kon\vdash\kon$ \\
\hline
4 & $\kon\vdash\kon$ & $\str\vdash$ \\
\hline
5 & $\str\vdash$ & $\str\vdash$ \\
\hline
6 & $\str\vdash$ & $\vdash\str$ \\
\hline
7 & $\vdash\str$ & $\vdash\str$ \\
\hline
\end{tabular}
\end{center}

\vspace{1ex}
\noindent\textit{Case 1.} Suppose that $\mathcal{D}_g$ ends with
\[
\f{g_1\colon\Pi_{B_1}\vdash B_1\quad g_2\colon\Pi_{B_2}\vdash B_2}{g_1\kon g_2\colon\Pi_{B_1}\kon \Pi_{B_2}\vdash B_1\kon B_2.}
\]
Since $f$ is $\mj_{B_1\kon B_2}$, which is equal to $\mj_{B_1}\kon \mj_{B_2}$, we can apply the induction hypothesis to $\mj_{B_1}, g_1$ and $\mj_{B_2}, g_2$ and obtain, for some central terms $\sigma_1$ and $\sigma_2$, that $f=g\circ (\sigma_1\kon \sigma_2)$. Since $\sigma_1\kon \sigma_2$ is central, we are done.

\vspace{2ex}
\noindent\textit{Case 2.} Suppose that $\mathcal{D}_g$ ends with
\[
\f{h\colon B_1\kon (B_1\str B_2)\vdash B_2}{(B_1\str h)\circ\eta_{B_1,(B_1\str B_2)}\colon B_1\str B_2\vdash B_1\str B_2.}
\]
Since $f$ is $\mj_{B_1\str B_2}$, which is equal to $(B_1\str \varepsilon_{B_1,B_2})\circ \eta_{B_1,(B_1\str B_2)}$, we can apply the induction hypothesis to $\varepsilon_{B_1,B_2}, h$ and obtain directly $f=g$.

\vspace{2ex}
\noindent\textit{Case 3.} Suppose that $\mathcal{D}_f$ and $\mathcal{D}_g$,  respectively, end with
\[
\frac{f_1\colon A_1\vdash B_1 \;\; f_2\colon A_2\vdash B_2\otimes B_3}{f_1\otimes f_2\colon A_1\otimes A_2\vdash B} \quad \frac{g_1\colon \Pi_{A_1^1\otimes A_2^1}\vdash B_1\otimes B_2 \;\; g_2\colon \Pi_{A_1^2\otimes A_2^2}\vdash B_3}{g_1\otimes g_2\colon\Pi_{A_1^1\otimes A_2^1}\otimes\Pi_{A_1^2\otimes A_2^2}\vdash B.}
\]
Since the sequences $A_1^2$ and $A_2^2, B_3$ are disjoint, by Proposition \ref{Lemma:Main1IL} there exists a derivation of $A_1^2\vdash I$, which, by Lemma \ref{Lemma:ConstProperIL}, implies that $A_1^2$ is constant. By the assumption based on Remark \ref{R3}, we conclude that $\mathcal{D}_g$ must be of the form:
\[
\frac{g_1\colon \Pi_{A_1\otimes A_2^1}\vdash B_1\otimes B_2 \quad g_2\colon \Pi_{A_2^2}\vdash B_3}{g_1\otimes g_2\colon\Pi_{A_1\otimes A_2^1}\otimes\Pi_{A_2^2}\vdash B.}
\]
Since the sequences $A_1,B_1$ and $A_2^1,B_2$ are disjoint, by Proposition~\ref{Lemma:Main1IL} there are derivations of $A_1\vdash B_1$ and $A_2^1\vdash B_2$ coded by $g_1'$ and $g_1''$, respectively. By the induction hypothesis $f_1=g_1'$, and for some central $\sigma_1, \sigma_2$, we have that $g_1'\kon g_1''=g_1 \circ \sigma_1$ and $f_2=(g_1''\kon g_2)\circ \sigma_2$. Hence,
\begin{tabbing}
$\quad f$ \= $=f_1\kon f_2=g_1'\kon((g_1''\kon g_2)\circ \sigma_2)=(g_1'\kon g_1''\kon g_2)\circ(A_1\kon\sigma_2)$
\\[.5ex]
\> $=(g_1\kon g_2)\circ(\sigma_1\kon \Pi_{A_2^2})\circ(A_1\kon\sigma_2)=g\circ \sigma$,
\end{tabbing}
for $\sigma=(\sigma_1\kon \Pi_{A_2^2})\circ(A_1\kon\sigma_2)$, and we are done.

\vspace{2ex}
\noindent\textit{Case 4.} Assume that the last rule of $\mathcal{D}_g$ introduces $\str$ in a formula corresponding to the one that belongs to the left-premise of the last rule of $\mathcal{D}_f$. We proceed analogously in the case when this formula belongs to the right premise of this rule. By analysing all situations, one concludes that the only possible scenario is when $\mathcal{D}_f$ ends with
\[
\f{f_1\colon \Pi_{A_1^1\kon(C\str D)\kon A_1^2}\vdash B_1\quad f_2\colon A_2\vdash B_2}{f_1\kon f_2\colon \Pi_{A_1^1\kon(C\str D)\kon A_1^2}\kon A_2\vdash B_1\kon B_2,}
\]
while $\mathcal{D}_g$ ends with
\[
\f{g_1\colon A_1^1\vdash C\quad g_2\colon D\kon \Pi_{A_1^2\kon A_2}\vdash B}{g_2\circ(\varepsilon\!\kon\!\Pi_{A_1^2\kon A_2})\circ (g_1\!\kon\!(C\str D)\!\kon\! \Pi_{A_1^2\kon A_2})\colon A_1^1\kon(C\str D)\kon\Pi_{A_1^2\kon A_2}\vdash B.}
\]
In this case $A$ is $\Pi_{A_1^1\kon(C\str D)\kon A_1^2}\kon A_2$, and $B$ is $B_1\kon B_2$.

From our assumptions on $A\vdash B$, we conclude that $A_1^1,C$ and $D,A_1^2, B_1$ are disjoint, and that $A_1^2,B_1$ and $A_2,B_2$ are disjoint. By Propositions \ref{Lemma:Main1IL} and \ref{Lemma:Main2IL} there are derivations $f_1'\colon A_1^1\vdash C$, $f_1''\colon D\kon A_1^2\vdash B_1$, $g_2'\colon D\kon A_1^2\vdash B_1$ and $g_2''\colon A_2\vdash B_2$.

From the induction hypothesis we have that $f_1'=g_1$, $f_2=g_2''$, $f_1''=g_2'$ and for some central terms $\sigma_1$ and $\sigma_2$
\[
f_1\circ \sigma_1=f_1''\circ(\varepsilon\kon A_1^2)\circ (f_1'\kon (C\str D)\kon A_1^2), \quad\quad g_2\circ \sigma_2=g_2'\kon g_2''.
\]
Since the type of $\sigma_2$ is $D\kon A_1^2\kon A_2\vdash D\kon \Pi_{A_1^2\kon A_2}$ and it is balanced, which means that $D\kon A_1^2\kon A_2$ is assorted, one can take arbitrary central term $\sigma_2'\colon A_1^2\kon A_2\vdash \Pi_{A_1^2\kon A_2}$, and by Corollary~\ref{Corollary:CentralTerms} conclude that $\sigma_2=D\kon \sigma_2'$. All this entails, in a straightforward manner, that $f=g\circ \sigma$, for some central~$\sigma$.

\vspace{2ex}
\noindent\textit{Case 5.1.} First we discuss the case when the same formula is introduced at the end of both derivations. Due to our assumptions on $A\vdash B$, the only possible situation is when $\mathcal{D}_f$ ends with
\[
\f{f_1\colon A_1\vdash C\quad  f_2\colon D\kon A_2\vdash B}{f_2\circ(\varepsilon\kon A_2)\circ(f_1\kon(C\str D)\kon A_2)\colon A_1\kon (C\str D)\kon A_2\vdash B,}
\]
while $\mathcal{D}_g$ ends with
\[
\f{g_1\colon \Pi_{A_1}\vdash C\quad  g_2\colon D\kon \Pi_{A_2}\vdash B}{g_2\circ(\varepsilon\kon \Pi_{A_2})\circ(g_1\kon(C\str D)\kon \Pi_{A_2})\colon \Pi_{A_1}\kon (C\str D)\kon \Pi_{A_2}\vdash B.}
\]
By the induction hypothesis, for some central terms $\sigma_1$ and $\sigma_2$, we have that $f_1=g_1\circ\sigma_1$ and $f_2=g_2\circ\sigma_2$. As in Case~4, we conclude that $\sigma_2=D\kon \sigma_2'$, for some central term $\sigma_2'\colon A_2\vdash \Pi_{A_2}$. All this entails that for some central $\sigma$ we have $f=g\circ\sigma$.

\vspace{2ex}
\noindent\textit{Case 5.2.} Next we have the case when two different formulae are introduced at the end of $\mathcal{D}_f$ and $\mathcal{D}_g$. This case has several variations depending on places where the connectives $\str$ are introduced. However, all of these are solved in a similar way. We will consider just the case when the last rule of $\mathcal{D}_g$ introduces $\str$ in a formula corresponding to the one that belongs to the right premise of the last rule of $\mathcal{D}_f$, while the last rule of $\mathcal{D}_f$ introduces $\str$ in a formula corresponding to the one that belongs to the left-premise of the last rule of $\mathcal{D}_g$. In this case, the only possible forms of $\mathcal{D}_f$ and $\mathcal{D}_g$ are such that $\mathcal{D}_f$ ends with
\[
\f{f_1\colon A_1\vdash C_1\quad f_2\colon D_1\kon\Pi_{A_2^1\kon(C_2\str D_2)\kon A_2^2}\vdash B}{f\colon A_1\kon (C_1\str D_1)\kon \Pi_{A_2^1\kon(C_2\str D_2)\kon A_2^2}\vdash B,}
\]
for $f$ being
\[
f_2\circ(\varepsilon\kon \Pi_{A_2^1\kon(C_2\str D_2)\kon A_2^2})\circ (f_1\kon (C_1\str D_1)\kon \Pi_{A_2^1\kon(C_2\str D_2)\kon A_2^2}),
\]
while $\mathcal{D}_g$ ends with
\[
\f{g_1\colon\Pi_{A_1\kon(C_1\str D_1)\kon A_2^1}\vdash C_2\quad g_2\colon D_2\kon\Pi_{A_2^2}\vdash B}{g\colon\Pi_{A_1\kon(C_1\str D_1)\kon A_2^1}\kon (C_2\str D_2)\kon \Pi_{A_2^2}\vdash B,}
\]
for $g$ being
\[
g_2\circ(\varepsilon\kon \Pi_{A_2^2})\circ (g_1\kon (C_2\str D_2)\kon \Pi_{A_2^2}).
\]

As in Case~4, from our assumptions on $A\vdash B$, we conclude that $D_1,A_2^1,C_2$ and $D_2,A_2^2, B$ are disjoint, and that $A_1,C_1$ and $D_1,A_2^1,C_2$ are disjoint. By Proposition \ref{Lemma:Main2IL}, there are derivations $f_2'\colon D_1\kon A_2^1\vdash C_2$, $f_2''\colon D_2\kon A_2^2\vdash B$, $g_1'\colon A_1\vdash C_1$ and $g_1''\colon D_1\kon A_2^1\vdash C_2$.

The induction hypothesis says that $f_1=g_1'$, $f_2'=g_1''$, $f_2''=g_2\circ \sigma_1$ and
\[
\begin{array}{c}
f_2\circ \sigma_2= f_2''\circ (\varepsilon\kon A_2^2)\circ (f_2'\kon(C_2\str D_2)\kon A_2^2),
\\[1ex]
g_1\circ \sigma_3= g_1''\circ (\varepsilon\kon A_2^1)\circ (g_1'\kon(C_1\str D_1)\kon A_2^1).
\end{array}
\]
In the same manner as above we conclude that $\sigma_1=D_2\kon\sigma_1'$ and $\sigma_2= D_1\kon\sigma_2'$, for central terms $\sigma_1'$ and $\sigma_2'$. Then starting with $f$ and substituting all $f$'s in it by $g$'s according to the above equalities, one obtains a term that transforms with the help of equality \ref{fun2} into $g\circ\sigma$, for some central $\sigma$.

\vspace{2ex}
\noindent\textit{Case 6.} Suppose that $\mathcal{D}_f$ ends with
\[
\f{f_1\colon A_1\vdash C\quad f_2\colon D\kon A_2\vdash B}{f_2\circ (\varepsilon\kon A_2)\circ (f_1\kon (C\str D)\kon A_2)\colon A_1\kon(C\str D)\kon A_2\vdash B,}
\]
and $\mathcal{D}_g$ ends with
\[
\f{g_1\colon B_1\kon\Pi_{A_1\kon(C\str D)\kon A_2}\vdash B_2}{(B_1\str g_1)\circ\eta\colon \Pi_{A_1\kon(C\str D)\kon A_2}\vdash B_1\str B_2.}
\]

\noindent As before, from our assumptions on $A\vdash B$, one concludes that $A_1,C$ and $B_1,D,A_2,B_2$ are disjoint. By Proposition~\ref{Lemma:Main2IL} there are derivations $g_1'\colon A_1\vdash C$ and $g_1''\colon D\kon B_1\kon A_2\vdash B_2$. By the induction hypothesis, $f_1=g_1'$, $f_2=(B_1\str (g_1''\circ (\mc\kon A_2)))\circ\eta_{B_1,D\kon A_2}$, and for some central term $\sigma_1$
\[
g_1\circ\sigma_1=g_1''\circ (\varepsilon\kon B_1\kon A_2)\circ(g_1'\kon (C\str D)\kon B_1\kon A_2).
\]
(Note that for the last equality we could apply the induction hypothesis due to the fact that we count the occurrences of $\str$ in the derived sequent twice and those of $\kon$ just once---since $g_1'$ and $g_1''$ are anonymous, we do not know whether the second component of the complexity pair decreases.)

Since the type of $\sigma_1$ is $A_1\kon(C\str D)\kon B_1\kon A_2\vdash B_1\kon \Pi_{A_1\kon(C\str D)}\kon A_2$, by appealing to the fact that $A_1\kon(C\str D)\kon B_1\kon A_2$ is assorted and relying on Corollary~\ref{Corollary:CentralTerms}, one concludes that for some central $\sigma_1'$
\[
\sigma_1=(B_1\kon\sigma_1')\circ (\mc_{A_1\kon(C\str D),B_1}\kon A_2).
\]
It remains to apply the procedure mentioned at the end of Case~5.2, save that besides \ref{fun2}, this time we have to rely on equalities \ref{nat}, \ref{iso}, \ref{fun2str} and \ref{nateta}.

\vspace{2ex}
\noindent\textit{Case 7.} Suppose that $\mathcal{D}_f$ and $\mathcal{D}_g$, respectively, end with
\[
\f{f_1\colon B_1\kon A\vdash B_2}{(B_1\str f_1)\circ\eta\colon A\vdash B_1\str B_2}\quad\quad \f{g_1\colon B_1\kon \Pi_{A}\vdash B_2}{(B_1\str g_1)\circ\eta\colon \Pi_{A}\vdash B_1\str B_2.}
\]
By the induction hypothesis, there is a central term $\sigma_1\colon B_1\kon A\vdash B_1\kon\Pi_{A}$ such that $f_1=g_1\circ\sigma_1$. As before, we may conclude that $\sigma_1$ is equal to $B_1\kon \sigma$ for a central $\sigma\colon A\vdash\Pi_A$. Hence,
\begin{tabbing}
$\quad f$ \= $=(B_1\str f_1)\circ\eta=(B_1\str (g_1\circ(B_1\kon\sigma)))\circ\eta$
\\[.5ex]
\> $=(B_1\str g_1)\circ (B_1\str(B_1\kon\sigma))\circ\eta= (B_1\str g_1)\circ\eta\circ\sigma=g\circ\sigma$,
\end{tabbing}
and we are done.
\end{proof}

\section{An application of the main result}\label{application}

The formulation of Theorem~\ref{main} is very restrictive with respect to derivations, or better to say with respect to derived sequents. However, there is just one easy step to transform it into a powerful machinery for detecting equal derivations in $\mathcal{IL}$. As we will see, in order to compare two derivations of the same sequent (not necessarily balanced, but definitely proper), one has to diversify as much as possible the variables (letters) occurring in these derivations without changing the rules and then to compare the derived sequents. If it is possible to make these diversifications so that the derived sequents are identical, then the initial derivations are equal.

Our presentation of the content of this section will be more practical than formal. We will avoid tedious difficulties with diversification and unification of variables and replace it with a graphical (or diagrammatical) technique that is rather natural and easy. Such a technique will bring us closer to a categorial formulation of the main result.

Let us assign \emph{links} to primitive $\mathcal{IL}$-terms $\mj_A$, $\mc_{B,A}$, $\eta_{A,B}$ and $\varepsilon_{A,B}$. The links connect a letter in one occurrence of $A$ in the type of a primitive term with the corresponding letter in another occurrence of $A$, and the same for $B$. For example, the primitive term $\eta_{p\kon p, q}$ induces the following links.

\centerline{\includegraphics[width=.4\textwidth]{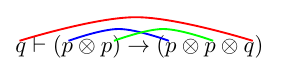} }

The links between pairs of letters in the type of an arbitrary term are formed in a similar manner. For practical reasons, we will write sequents so that the antecedent is above the consequent and then draw the links. The above example is then illustrated as:

\centerline{\includegraphics[width=.5\textwidth]{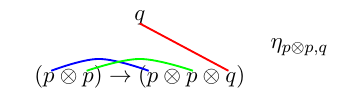} }

\noindent With this convention in mind, the links of $f\kon g$ are the links of $f$ and the links of $g$ placed side by side, while the links of $A\str f$ are the links of $\mj_A$ and the links of $f$ placed side by side. The operation $\circ$ on terms corresponds to the ``gluing'' operation on links. This will become more formal in Section~\ref{dictionary} (see the end of the part on cobordisms) when we assign to the links a pure mathematical meaning. A sufficiently illustrative example is given for the term $(p\str(\varepsilon_{p,q}\kon\mj_p))\circ\eta_{p,(p\str q)\kon p}\circ \mc_{p,p\str q}$ in Figure~\ref{links}.
\begin{figure}[!h]
    \centerline{\includegraphics[width=0.9\textwidth]{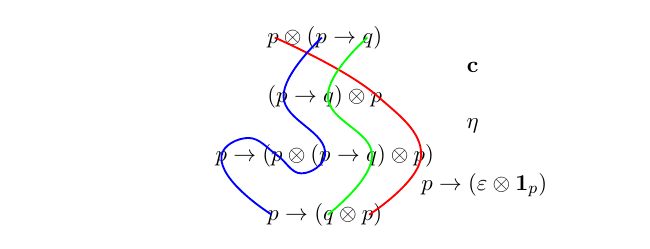} } \caption{Links induced by a term}\label{links}
\end{figure}

\begin{lem}\label{dodatak3}
Two equal terms induce the same links.
\end{lem}

\begin{proof}
By inspecting the equalities \ref{cat1}-\ref{triang2}. The mathematical content of this proof is a well known fact that the category of 1-dimensional cobordisms is symmetric monoidal closed (for these notions see Section~\ref{dictionary}).
\end{proof}

A \emph{diversification} of a term is obtained by changing its indices so that two letters in the type are equal only if they are connected by a link. For example, a diversification of $\eta_{p\kon p, q}\colon q\vdash (p\kon p)\str(p\kon p\kon q)$ is $\eta_{r\kon s, q}\colon q\vdash (r\kon s)\str(r\kon s\kon q)$, while $\eta_{p\kon p, q}$ is a \emph{unification} of $\eta_{r\kon s, q}$.

\begin{lem}
Every term has a diversification.
\end{lem}

\begin{proof}
Let $f$ be a term. We proceed by induction on the complexity of $f$. If $f$ is primitive, then change its index so that every letter occurs just once in it. If $f$ is of the form $f'\kon f''$, then we just apply the induction hypothesis to $f'$ and $f''$, taking care that the chosen diversifications are ``disjoint''. We proceed in the same manner in the case when $f$ is of the form $A\str f'$, when the induction hypothesis is applied to $\mj_A$ and $f'$.

If $f$ is of the form $f_2\circ f_1$, then we apply the induction hypothesis to $f_1$ and $f_2$ obtaining their diversifications. Then we make unifications of the obtained terms so that the links induced by $f_1$ and the links induced by $f_2$ participating in one link induced by $f$ share the same letter.
\end{proof}

In the example from Figure~\ref{links}, let us make a diversification of the composition $\eta_{p,(p\str q)\kon p}\circ \mc_{p,p\str q}$ first. Start with a diversification $\mc_{p,r\str q}$ of $\mc_{p,p\str q}$ and a diversification $\eta_{s,(r\str q)\kon p}$ of $\eta_{p,(p\str q)\kon p}$, which are chosen so that the links participating in the same link of the composition share the same letter, we make a diversification $\eta_{s,(r\str q)\kon p}\circ\mc_{p,r\str q}\colon p\kon(r\str q)\vdash s\str(s\kon (r\str q)\kon p)$ of this composition.

In order to make a diversification of the whole term, the above diversification and a diversification $u\str(\varepsilon_{v,q}\kon\mj_p)$ of $p\str(\varepsilon_{p,q}\kon\mj_p)$ should be unified so that the links participating in the links induced by the whole term share the same letter. This could be done by taking the unifications $\eta_{r,(r\str q)\kon p}\circ\mc_{p,r\str q}$ and $r\str(\varepsilon_{r,q}\kon\mj_p)$, and finally we obtain a diversification
\[
(r\str(\varepsilon_{r,q}\kon\mj_p))\circ \eta_{r,(r\str q)\kon p}\circ\mc_{p,r\str q}\colon p\kon(r\str q)\vdash r\str(q\kon p)
\]
of the whole term.

\begin{rem}\label{dodatak4}
A diversification of a term is balanced. Two terms of the same type induce the same links if and only if they have diversifications of the same type.
\end{rem}

\begin{lem}\label{dodatak5}
If diversifications of two terms of the same type are equal, then these two terms are equal.
\end{lem}

\begin{proof}
Let $f^d,g^d\colon A^d\vdash B^d$ be diversifications of $f,g\colon A\vdash B$. Note that a term and its diversification differ only in their indices and one may transform a proof of $f^d=g^d$ (just by changing indices) into a proof of $f=g$.
\end{proof}

All from above leads to the following result.
\begin{prop}\label{coherence2}
Two terms of the same proper type are equal if and only if they induce the same links in the common type.
\end{prop}

\begin{proof}
If two terms are equal, then by Lemma~\ref{dodatak3} they induce the same links. For the other direction, let $A\vdash B$ be proper and let $f,g\colon A\vdash B$ induce the same links. By Remark~\ref{dodatak4}, these terms have balanced diversifications $f^d,g^d\colon A^d\vdash B^d$. Since $A\vdash B$ is proper, $A^d\vdash B^d$ is proper too. By Theorem~\ref{main}, we have that $f^d=g^d$ and by Lemma~\ref{dodatak5} the terms $f$ and $g$ are equal.
\end{proof}

This approach enables us to check equality of arbitrary $\mathcal{IL}$-terms of the same proper type. If one prefers to keep to the terms coding $\mathcal{IL}$-derivations, then the links could be drawn directly in the sequent derivations of this system. For example, let $\mathcal{D}_1$ and $\mathcal{D}_2$ be two $\mathcal{IL}$-derivations coded by terms $f_1,f_2\colon p\kon(p\str q)\vdash p\str(q\kon p)$, where
\[
f_1=(p\str((\varepsilon\kon\mj_p)\circ (p\kon\mc)\circ (\mc\kon(p\str q))))\circ\eta,
\]
while
\[
f_2=(p\str(\varepsilon\kon\mj_p))\circ\eta\circ \mc.
\]
The links induced by $f_1$ and $f_2$ are drawn directly in the derivations  $\mathcal{D}_1$ and $\mathcal{D}_2$ as follows (the left-hand side derivation is $\mathcal{D}_1$):

\centerline{\includegraphics[width=1.3\textwidth]{d1.pdf} }

We claim that these two derivations are not equivalent since the links induced by $f_1$ and $f_2$ are the following.

\centerline{\includegraphics[width=1.3\textwidth]{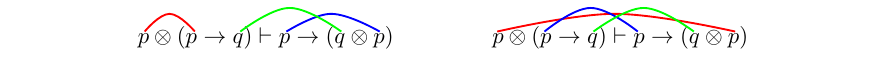} }

One can come closer to the notion of 1-dimensional cobordism that will be discussed in Section~\ref{dictionary} by replacing letters in the types by their signs ($+$ if the occurrence is positive and $-$ if it is negative). In the example above, we have the following.

\centerline{\includegraphics[width=1.3\textwidth]{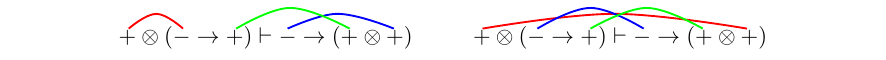} }

\begin{rem}[On generalization conjecture]\label{dodatak6}
By the generalization conjecture (see the Introduction) two derivations with the same premisses and conclusions are considered equivalent when for every maximal generalization of one of them there exists a maximal generalization of the other such that both have the same premisses and conclusions. Let $\mathcal{D}$ be an $\mathcal{IL}$-derivation coded by a term $f$. By the definition of diversification it follows that a maximal generalization of $\mathcal{D}$ is the derivation coded by a diversification of $f$. Together with Remark~\ref{dodatak4}, this entails that two derivations with the same premisses and conclusions are coded by two terms inducing the same links in the common type if and only if these two derivations have maximal generalizations with the same premisses and conclusions. By Proposition~\ref{coherence2}, if the common type is proper, these terms are equal and we conclude that if we restrict to proper formulae, the generality conjecture holds for $\mathcal{IL}$.
\end{rem}

\section{An elementary dictionary of category theory}\label{dictionary}

This section contains some elementary categorial notions having their proof-theo\-retical counterparts in the preceding text. We will not go into complete details. Except for the notion of cobordism, for which we suggest to consult \cite[Section~1.2]{K03}, the other features could be found in \cite{ML98}. The notions are listed alphabetically, and not in the order of appearance in our text. Every notion has its reference to the corresponding part covered before.

\vspace{1ex}
\noindent \textbf{Adjunctions.} Given two categories $\mathcal{C}$ and $\mathcal{D}$, an
\emph{adjunction} is given by two functors, $F\!:\mathcal{C}\str
\mathcal{D}$ and $G\!:\mathcal{D}\str \mathcal{C}$, and two
natural transformations, the \emph{unit}
$\eta\colon\mj_\mathcal{C}\strt GF$ and the \emph{counit}
$\varepsilon\colon FG\strt \mj_\mathcal{D}$, such that for every $C\in
O_\mathcal{C}$ and every $D\in O_\mathcal{D}$
\[
G\varepsilon_D \circ \eta_{GD}=\mj_{GD}, \quad \mbox{\rm and}\quad
\varepsilon_{FC}\circ F\eta_C=\mj_{FC}.
\]
The functor $F$ is a \emph{left adjoint} for the functor $G$, while
$G$ is a \emph{right adjoint} for the functor $F$.

The equalities \ref{triang1} and \ref{triang2} together with \ref{nateta} and \ref{nateps} say that $A\kon\underline{\mbox{\hspace{.7em}}}$ is a left adjoint for $A\str\underline{\mbox{\hspace{.7em}}}$.

\vspace{1ex}
\noindent \textbf{Categories.} A \emph{category} (within set theory) consists of two sets, $O$ of \emph{objects} and
$A$ of \emph{arrows}, and two functions
\[
\mbox{\rm source, target}\!:A\str O.
\]
For every object $X$ there is an arrow $\mj_X$ and for every pair $f$, $g$ of arrows such that $\mbox{\rm
source}(g)=\mbox{\rm target}(f)$ there is an arrow $g\circ f$, the \emph{composition} of $f$ and $g$.
Moreover, for
all $X,Y\in O$, and all $f,g,h\in A$ such that $h\circ g$,
$g\circ f$ and $\mj_Y\circ f$ are defined, the
following holds
\[
\mbox{\rm source}(\mj_X)=X=\mbox{\rm target}(\mj_X),
\]
\[
\mbox{\rm source}(g\circ f)=\mbox{\rm source}(f),\quad \mbox{\rm
target}(g\circ f)=\mbox{\rm target}(g),
\]
and
\[
g\circ\mj_Y=g,\quad \mj_Y\circ f=f,\quad h\circ(g\circ f)=(h\circ
g)\circ f.
\]

The equalities \ref{cat1} and \ref{cat2} say that the formulae of $\mathcal{IL}$ make the set of objects of a category whose arrows are the terms modulo equality relation. We denote this category also by $\mathcal{IL}$. For $f\colon A\vdash B$, the source of $f$ is $A$ and its target is $B$.

\vspace{1ex}
\noindent \textbf{Cobordisms.} A 1-dimensional cobordism is a triple \[(M,f_0\colon a\to M, f_1\colon b\to M),\] where $M$ is a compact oriented 1-dimensional manifold possibly with boundary (i.e.\ a finite collection of oriented circles and line segments), $a$ and $b$ are two finite collections of oriented points and $f_0$, $f_1$ are embeddings. If the boundary of $M$ is $\Sigma_0\coprod \Sigma_1$ and its orientation is induced from the orientation of $M$ (the initial point of an oriented segment is $+$ while the terminal is $-$), then $f_0\colon a\to \Sigma_0$ is orientation preserving, while $f_1\colon b\to \Sigma_1$ is orientation reversing. The source of such a cobordism is $a$ while $b$ is its target.

Two cobordisms $(M,f_0,f_1)$ and $(M',f'_0,f'_1)$ are
\emph{equivalent}, when there is an
orientation preserving homeomorphism $F:M\to M'$ such that the
following diagram commutes.

\begin{center}
\begin{tikzcd}
a\arrow[hook]{r}{f_0}\arrow[hook]{dr}[swap]{f'_0}
&M \arrow{d}{F} \arrow[hookleftarrow]{r}{f_1} & b \arrow[hook]{ld}{f'_1}\\
&M'
\end{tikzcd}
\end{center}

The category 1Cob has closed oriented 0-manifolds (i.e.\ finite collections of oriented points) as objects and equivalence classes of cobordisms as arrows. The composition of cobordisms is defined by ``gluing'', or more formally by making pushouts of pairs of embeddings. (For the notion of pushout see \cite[Section II.3]{ML98}.)

A formal interpretation of our diagrammatical approach in Section~\ref{application} is that we have defined there a correspondence between $\mathcal{IL}$-derivations (or just terms) and 1-dimensional cobordisms. Our links between pairs of variables in derived sequents are nothing but 1-dimensional cobordisms and the operation $\circ$ on terms corresponds to the composition of cobordisms. The case of the operation $\circ$ in the definition of links assigned to a term (see Section~\ref{application}) relies on the fact that the composition of two cobordisms is again a cobordism.

\vspace{1ex}
\noindent \textbf{Coherence.} Mac Lane, \cite[Section~3]{ML63} introduced the word ``coherent'' for the following property of the natural transformation $\alpha$ defined as in a symmetric monoidal category: $\alpha$ is \emph{coherent} when for each pair of functors obtained by iterating $\kon$, there is at most one iterate of $\alpha$ which is a natural isomorphism between them. The term \emph{coherence} stems from this definition. In its most primitive form, like in the case of $\alpha$, coherence is formulated as ``all diagrams commute''.

Many different results from many different fields of mathematics may be treated as coherence. We mention just a few: \cite[first statement of Proposition~3]{S63}, \cite[Theorem~2.4]{KML71}, \cite[XI.3, Theorem~1]{ML98}, \cite[Proposition~4]{L68}, \cite[Theorems~1-2]{S90} \cite[Theorem~3.6]{BFSV},  \cite[Theorem~2.5]{K93}, \cite[Theorem]{S97}, \cite[Theorem~2.5]{T10} and ~\cite{S09}. Our Theorem~\ref{main} has the same mathematical content as \cite[Theorem~2.4]{KML71}.

\vspace{1ex}
\noindent \textbf{Functors.} Given two categories $\mathcal{C}$ and $\mathcal{D}$, a
\emph{functor} $F\colon \mathcal{C}\str\mathcal{D}$ consists of two
functions, both denoted by $F$, the \emph{object function}
$F\!:O_\mathcal{C}\str O_\mathcal{D}$ and the \emph{arrow
function} $F\!:A_\mathcal{C}\str A_\mathcal{D}$, such that for
every $C\in O_\mathcal{C}$ and every composition $g\circ f$ of
arrows of $\mathcal{C}$, the following holds
\[
F\mj_C=\mj_{FC},\quad F(g\circ f)=Fg\circ Ff.
\]
A functor $F\colon \mathcal{C}\str\mathcal{D}$ is \emph{faithful} when for all $f,g\colon A\vdash B$ in $\mathcal{C}$ if $Ff=Fg$, then $f=g$.

The equalities \ref{fun1} and \ref{fun2} say that $\kon$ is a functor from $\mathcal{IL}\times \mathcal{IL}$ to $\mathcal{IL}$, where the structure of $\mathcal{IL}\times \mathcal{IL}$ is defined componentwise. The equalities \ref{fun1str} and \ref{fun2str} say that $A\str\underline{\mbox{\hspace{.7em}}}$ is a functor from $\mathcal{IL}$ to $\mathcal{IL}$. Most of the cases in our cut-elimination procedure (see the proof of Theorem~\ref{cutelim1}) use the functoriality of the connectives. Proposition~\ref{coherence2} says that the correspondence between terms and links underlies a faithful functor from $\mathcal{IL}$ (restricted to proper objects) to the category 1Cob of 1-dimensional cobordisms.

\vspace{1ex}
\noindent \textbf{Isomorphisms.} An arrow $f\colon A\vdash B$ of a category $\mathcal{C}$ is an
\emph{isomorphism} when there is an arrow $g\colon B\vdash A$ in
$\mathcal{C}$, such that $g\circ f=\mj_A$ and $f\circ g=\mj_B$ in
$\mathcal{C}$. We say that $g$ is the \emph{inverse} of $f$.

The equality \ref{iso} says that $\mc_{A,B}$ is an isomorphism with $\mc_{B,A}$ as its inverse. Definition~\ref{invertible} says that invertible terms represent isomorphisms in $\mathcal{IL}$. Remarks~\ref{ItoI} and \ref{R3}, Lemma~\ref{AtoI} and Proposition~\ref{R1} are also about isomorphisms.

\vspace{1ex}
\noindent \textbf{Natural transformations.}
Given two functors $F,G\!:\mathcal{C}\str\mathcal{D}$, a
\emph{natural transformation} $\alpha\!:F\strt G$ is a function
from $O_\mathcal{C}$ to $A_\mathcal{D}$, i.e.,\ a family of arrows
of $\mathcal{D}$ indexed by the objects of $\mathcal{C}$, such
that for every $C\in O_\mathcal{C}$, $\mbox{\rm
source}(\alpha_C)=FC$ and $\mbox{\rm target}(\alpha_C)=GC$, and
for every $f\!:C\str C'\in A_\mathcal{C}$, the following diagram commutes in $\mathcal{D}$.
\[
\begin{CD}
FC @>{\alpha_C}>> GC \\
@V{Ff}VV  @VV{Gf}V \\
FC' @>>{\alpha_{C'}}> GC' \\
\end{CD}
\]
If all the elements of the family are isomorphisms, then this is a \emph{natural isomorphism}.

The equalities \ref{nat} and \ref{iso} say that $\mc$ is a natural isomorphism from $\kon$ to $\kon$ (precomposed by a transposition of arguments). The equalities \ref{nateta} and \ref{nateps} say that $\eta$ is a natural transformation from the identity functor on $\mathcal{IL}$ (the identity on arrows and objects) to the functor $A\str(A\kon\; )$, and $\varepsilon$ is a natural transformation from the functor $A\kon(A\str\; )$ to the identity functor on $\mathcal{IL}$. In some cases of our cut-elimination procedure (see the proof of Theorem~\ref{cutelim1}) we use the naturality of $\mc$, $\eta$ and $\varepsilon$.

\vspace{1ex}
\noindent \textbf{Symmetric monoidal categories.} A category $\mathcal{C}$ is \emph{symmetric monoidal} when it is equipped with a functor $\kon\colon\mathcal{C}\times
\mathcal{C}\str \mathcal{C}$ and an object $I$
such that the following holds. There are three natural isomorphisms with components
\[
\alpha_{A,B,C}\!:A\otimes(B\otimes C)\str (A\otimes B)\otimes C,
\lambda_A\!:I\otimes A\str A, \mc_{A,B}\!:A\otimes B\str
B\otimes A.
\]
The natural isomorphism $\mc$ is \emph{self-inverse},
i.e.,\ $\mc_{B,A}\circ\mc_{A,B}=\mj_{A\otimes B}$. Moreover, the following diagrams (\emph{coherence conditions}) commute.

\begin{equation}
\label{pentagon}
\minCDarrowwidth5pt
\begin{tikzcd}[column sep=small]
A\otimes(B\otimes(C\otimes D)) \arrow{r}{\alpha}
\arrow{d}[swap]{\textbf{1}\otimes\alpha} &(A\otimes
B)\otimes(C\otimes D) \arrow{r}{\alpha}
&((A\otimes B)\otimes C)\otimes D \\
A\otimes((B\otimes C)\otimes D) \arrow{r}[swap]{\alpha}
&(A\otimes(B\otimes C))\otimes D \arrow{ur}[swap]{\alpha\otimes
\textbf{1}} &
\end{tikzcd}
\end{equation}

\[
\begin{tikzcd}[column sep=small]
(I\otimes A)\otimes B \arrow{dr}{\lambda\otimes \textbf{1}}
\\
I\otimes (A\otimes B) \arrow{u}{\alpha} \arrow{r}[swap]{\lambda} &
A\otimes B
\end{tikzcd}
\hspace{.5em}
\begin{tikzcd}[column sep=small]
A\otimes(B\otimes C) \arrow{r}{\alpha}
\arrow{d}[swap]{\textbf{1}\otimes\mc} &(A\otimes B)\otimes C
\arrow{r}{\mc}
&C\otimes(A\otimes B) \arrow{d}{\alpha}\\
A\otimes(C\otimes B) \arrow{r}[swap]{\alpha} &(A\otimes C)\otimes
B \arrow{r}[swap]{\mc\otimes \textbf{1}} &(C\otimes A)\otimes B
\end{tikzcd}
\]

A symmetric monoidal category is \emph{strict
monoidal} when all the arrows in $\alpha$ and $\lambda$
are identities.

The equalities \ref{cat1}-\ref{coh} say that $\mathcal{IL}$ is a symmetric strict monoidal category. The loss of generality made by passing to strict monoidal structure is negligible by \cite[Section XI.3, Theorem~1]{ML98}. However the non-strict monoidal structure given by families $\alpha$ and $\lambda$ is combinatorially interesting and deserves a particular attention. We have skipped it on this occasion in order to make our proofs in Sections \ref{IL} and~\ref{main section} less complicated.

\vspace{1ex}
\noindent \textbf{Symmetric monoidal closed categories.} A symmetric monoidal category $\mathcal{C}$ is (monoidal) \emph{closed} when for every object $A$ it is equipped with a functor
$A\str\underline{\mbox{\hspace{.7em}}}\colon\mathcal{C}\to \mathcal{C}$, which is a right adjoint for the functor $A\kon\underline{\mbox{\hspace{.7em}}}\colon\mathcal{C}\to \mathcal{C}$.

The category of sets and functions is symmetric monoidal closed. In this case the Cartesian product serves as $\kon$, while $X\str Y$ is interpreted as the set of functions from $X$ to $Y$. For any field $K$, the category of vector spaces over $K$ is symmetric monoidal closed. The usual tensor product serves as $\kon$, and $V\str W$ is interpreted as the vector space of linear transformations from $V$ to $W$. The category 1Cob of cobordisms is also symmetric monoidal closed. For its objects $X$ and $Y$, the object $X\kon Y$ is just the disjoint union $X\sqcup Y$, while $X\str Y$ is again the disjoint union $X^\ast\sqcup Y$, where $\ast$ replaces + by - and vice versa. There are much more examples of such categories.

The equalities \ref{cat1}-\ref{triang2} say that $\mathcal{IL}$ is a symmetric (strict) monoidal closed category. It is freely generated by the set of propositional letters. The universal property of $\mathcal{IL}$ is the following: for every function $f$ from the set of propositional letters to the set of objects of an arbitrary symmetric strict monoidal closed category $\mathcal{C}$, there is a unique functor from $\mathcal{IL}$ to $\mathcal{C}$ that extends $f$ and preserves the symmetric monoidal closed structure. This property enables us to find models of $\mathcal{IL}$ in the rest of mathematics. These models rarely serve for checking derivability relation---they serve to check the equality of derivations in this logic. A deeper analysis of the full coherence conditions for symmetric monoidal closed categories is present in \cite{V77}, \cite{S90}, \cite{S97} and \cite{MS07}.

\appendix

\section{}\label{sec:appendix}

We have to introduce some notions before we proceed to the proof of Theorem~\ref{cutelim1}. By an $\alpha$-formula in an $\mathcal{IL}$-derivation we always mean its particular occurrence in this derivation. The sequent $C\vdash A$ is the \emph{left-premise} and $G\kon A\kon E\vdash D$ is the \emph{right premise} of the cut inference figure.

For every inference figure of $\mathcal{IL}$, except $\vdash\str$ and $\kon\vdash\kon$, the consequent $D$ of the lower sequent has the unique \emph{successor}---an occurrence of the same $\alpha$-formula $D$ as the consequent of an upper sequent. When the rule $\kon\vdash\kon$ is in question, then if one of $C$ or $D$ is $I$ and the other is not, the consequent of the lower sequent has the unique successor in the same sense as above. If both $C$ and $D$ are $I$, then the consequent of the lower sequent has two successors, the consequents of the upper sequents of this rule.

On the other hand, every prime factor of the antecedent of the lower sequent of an inference figure, except $A\str B$ in $\str\vdash$, has the unique \emph{successor}, an occurrence of the same $\alpha$-formula in the antecedent of an upper sequent. Let the \emph{rank} of the consequent $D$ of a sequent in a derivation be the number of $\alpha$-formulae of that derivation that are related to $D$ by the reflexive and transitive closure of the successor relation, and let the \emph{rank} of a prime factor of the antecedent of a sequent in a derivation be defined in the same manner. Every non-prime factor of the antecedent of a sequent has \emph{rank} 1.

Let the $\alpha$-formula $A$ in the cut inference figure be called \emph{cut formula}. Let the \emph{degree of a cut} in a derivation be the number of occurrences of $\str$ and $\kon$ in the cut formula~$A$. Let the \emph{rank of a cut} in a derivation be the sum of the rank of the cut formula in the left-premise and the rank of the cut formula in the right premise of this cut inference figure, save that if the cut formula is $I$, the rank of the cut formula in the right premise (which is usually invisible) is always~1.

\begin{proof}[Proof of Theorem~\ref{cutelim1}]
It suffices to prove the case when the derivation coded by $f$ has cut
as the last inference figure and there is no other application of cut in this derivation. As usually with cut-elimination procedures, we proceed by induction on lexicographically ordered
pairs $(d,r)$, where $d$ is the degree and $r$ is the rank of the
cut in such a derivation.

\vspace{1ex}
\noindent (0) For the basis, when $(d,r)=(0,2)$, the derivation is of the form
\[
\f{\mj_p\colon p\vdash p\quad \mj_p\colon p\vdash p}{\mj_p\circ\mj_p\colon p\vdash p} \quad\mbox{\rm or}\quad \f{\afrac{\mj_I\colon I\vdash I}\quad \fp{\mathcal{D}}{g\colon G\otimes E\vdash D}}{g\circ(G\otimes\mj_I\otimes E)\colon G\otimes E\vdash D}
\]
and we transform the first into the derivation consisting only of the
axiomatic sequent $\mj_p\colon p\vdash p$ (by \ref{cat1} we have that $\mj_p\circ\mj_p=\mj_p$). The second derivation is transformed into
\[
\fp{\mathcal{D}}{g\colon G\otimes E\vdash D,}
\]
and by \ref{fun1} and \ref{cat1}, we have that $g\circ(G\otimes\mj_I\otimes E)=g$.

\vspace{2ex}
\noindent (1) When $d>0$ and $r=2$ we have the following cases.

\vspace{1ex}
(1\emph{a}) Derivations of the
following forms
\[
\f{\afrac{\mj_A\colon A\vdash A}\quad \fp{\mathcal{D}}{f\colon G \kon A\kon E\vdash
D}}{f\circ (G\kon\mj_A\kon E)\colon G\kon A\kon E\vdash D}\quad\quad\quad
\f{\fp{\mathcal{D}}{g\colon C\vdash A}\quad \afrac{\mj_A\colon A\vdash
A}}{\mj_A\circ g\colon C\vdash A}
\]
are transformed, respectively, into the following cut-free
derivations
\[
\fp{\mathcal{D}}{f\colon G \kon A\kon E\vdash
D}\quad\quad\quad
\fp{\mathcal{D}}{g\colon C\vdash A.}
\]
The last two terms are by \ref{fun1} and \ref{cat1} equal, respectively, to the two terms above.

\vspace{1ex}
(1\emph{b}) If the derivation is of the form
\[
\f{\f{\fp{\mathcal{D}_1}{f\colon A_1\kon C\vdash A_2}}{C\vdash
A_1\str A_2}\quad\quad \f{\fp{\mathcal{D}_2}{g\colon G\vdash
A_1}\quad \fp{\mathcal{D}_3}{h\colon A_2\kon E\vdash D}}{G\kon (A_1\str
A_2)\kon E\vdash D}}{G\kon C\kon E\vdash D}
\]
and it is coded by the term
\[
u=h\circ(\varepsilon\kon E)\circ(g\kon(A_1\str A_2)\kon E)\circ (G\kon ((A_1\str f)\circ\eta)\kon E),
\]
then this derivation is transformed into the following derivation
\begin{equation}\label{3}
\f{\fp{\mathcal{D}_2}{g\colon G\vdash
A_1}\rzb\f{\fp{\mathcal{D}_1}{f\colon A_1\kon C\vdash A_2}\rza
\fp{\mathcal{D}_3}{h\colon A_2\kon E\vdash D}}{h\circ (f\kon E)\colon A_1\kon C\kon E\vdash
D}}{h\circ(f\kon E)\circ (g\kon C\kon E)\colon G\kon C\kon E \vdash D.}
\end{equation}
We have the following calculation:
\begin{tabbing}
$u$ \=$=h\circ((\varepsilon\circ(g\kon(A_1\str A_2))\circ(G\kon(A_1\str f))\circ(G\kon \eta))\kon E)$, \= by \ref{fun2}
\\[.5ex]
\> $=h\circ((\varepsilon\circ(A_1\kon(A_1\str f))\circ(A_1\kon \eta)\circ(g\kon C))\kon E)$, \> by \ref{fun2}, \ref{nateta}
\\[.5ex]
\> $=h\circ((f\circ\varepsilon\circ(A_1\kon \eta)\circ(g\kon C))\kon E)$, \> by \ref{nateps}
\\[.5ex]
\> $=h\circ((f\circ(g\kon C))\kon E)$, \> by \ref{triang1}
\\[.5ex]
\> $=h\circ(f\kon E)\circ(g\kon C\kon E)$. \> by \ref{fun2}
\end{tabbing}

The upper cut in the derivation (\ref{3}) is of lower degree than the original cut. Hence, by the induction hypothesis, one can find a cut-free derivation of $A_1\kon C\kon E\vdash D$ coded by a term $v$ equal to $h\circ(f\kon E)$. So, we have the following derivation, where $\mathcal{D}_4$ is cut-free
\begin{equation}\label{3.1}
\f{\fp{\mathcal{D}_2}{g\colon G\vdash
A_1}\rzb \fp{\mathcal{D}_4}{v\colon A_1\kon C\kon E\vdash
D}}{v\circ (g\kon C\kon E)\colon G\kon C\kon E \vdash D.}
\end{equation}
Since the cut in (\ref{3.1}) is of lower degree than the original, one may apply the induction hypothesis again, in order to obtain a cut-free derivation of $G\kon C\kon E\vdash D$ coded by a term equal to $v\circ (g\kon C\kon E)$ and hence to $u$.

\vspace{1ex}
(1\emph{c}) Finally, if the derivation is of the form
\[
\f{\f{\fp{\mathcal{D}_1}{f\colon C_1\vdash A_1}\quad \fp{\mathcal{D}_2}{g\colon C_2\vdash A_2}}{C_1\kon C_2\vdash
A_1\kon A_2}\quad\quad \fp{\mathcal{D}_3}{h\colon G\kon A_1\kon A_2\kon E\vdash
D}}{G\kon C_1\kon C_2\kon E\vdash D}
\]
and it is coded by the term $u=h\circ(G\kon f\kon g\kon E)$, then since $r=2$, by our definition of rank, neither $A_1$ nor $A_2$ is $I$. This derivation is transformed into the following derivation
\[
\f{\fp{\mathcal{D}_2}{g\colon C_2\vdash A_2}\rzb\f{\fp{\mathcal{D}_1}{f\colon C_1\vdash A_1}\rza
\fp{\mathcal{D}_3}{h\colon G\kon A_1\kon A_2\kon E\vdash D}}{h\circ(G\kon f\kon A_2\kon E)\colon G\kon C_1\kon A_2\kon E\vdash
D}}{h\circ(G\kon f\kon A_2\kon E)\circ(G\kon C_1\kon g\kon E)\colon G\kon C_1\kon C_2\kon E \vdash D}
\]

By \ref{fun2}, we have that $u=h\circ(G\kon f\kon A_2\kon E)\circ(G\kon C_1\kon g\kon E)$. We proceed with two remaining cuts as in (1\emph{b}).

\vspace{2ex}
\noindent (2) When $r>2$, either the derivation of the left-premise of the cut ends with one of the following inference figures
\begin{equation}\label{4}
\f{C_1\kon C_2\kon C_3\kon C_4\vdash A}{C_1\kon C_3\kon C_2\kon C_4\vdash
A}\pravilo{\rm int.}\rzb\rzb \f{C_1\vdash C_2\rzb C_3\kon C_4\vdash A}{C_1\kon (C_2\str
C_3)\kon C_4\vdash A}\pravilo{$\str\vdash$}
\end{equation}

\begin{equation}\label{4a}
\f{C_1\vdash A\quad C_2\vdash I}{C_1\kon C_2\vdash
A}\pravilo{$\otimes\vdash\otimes$}\rzb\rzb\rzb \f{C_1\vdash I\quad C_2\vdash I}{C_1\kon C_2\vdash I}\pravilo{$\otimes\vdash\otimes$,}
\end{equation}
or, for $A$ being a prime factor in the antecedent of the right premise of the cut, the derivation of this premise ends with one of
the following inference figures (in the right-hand sides of (\ref{5}) and (\ref{6}) below, $A$ is a prime factor of $G'$ or~$E'$).
\begin{equation}\label{5}
\f{G'\kon A\kon E'\vdash D}{G\kon A\kon E\vdash
D}\pravilo{\rm interchange}
\mbox{\hspace{8em}}
\f{G'\vdash X \quad Y\kon E'\vdash D}{G\kon A\kon E\vdash
D}\pravilo{$\str\vdash$}
\end{equation}

\begin{equation}\label{6}
\mbox{\hspace{-1.3em}}\f{D_1\kon G\kon A\kon E\vdash D_2}{G\kon A\kon E\vdash D_1\str
D_2}\pravilo{$\vdash\str$} \mbox{\hspace{4.3em}} \f{G'\vdash D_1\quad E'\vdash D_2}{G\kon A\kon E\vdash D_1\kon D_2}\pravilo{$\kon\vdash\kon$}
\end{equation}

(2\emph{a}) If the derivation is of the form
\[
\f{\f{\fp{\mathcal{D}_1}{f\colon C_1\kon C_2\kon C_3\kon C_4\vdash A}}{f\circ(C_1\kon\mc\kon C_4)\colon C_1\kon C_3\kon C_2\kon C_4\vdash A}\quad\quad \fp{\mathcal{D}_3}{g\colon G\kon A\kon E\vdash
D}}{g\circ(G\kon(f\circ(C_1\kon\mc\kon C_4))\kon E)\colon G\kon C_1\kon C_3\kon C_2\kon C_4\kon E\vdash D,}
\]
then this derivation is transformed into the following derivation:
\[
\f{\f{\fp{\mathcal{D}_1}{f\colon C_1\kon C_2\kon C_3\kon C_4\vdash A}\quad\quad \fp{\mathcal{D}_3}{g\colon G\kon A\kon E\vdash
D}}{g\circ(G\kon f\kon E)\colon G\kon C_1\kon C_2\kon C_3\kon C_4\kon E\vdash D}}{g\circ(G\kon f\kon E)\circ(G\kon C_1\kon\mc\kon C_4\kon E)\colon G\kon C_1\kon C_3\kon C_2\kon C_4\kon E\vdash D.}
\]

By \ref{fun2}, the terms coding these two derivations are equal. Since the rank of the new cut is lower than the rank of the original, and the degrees are the same, one may apply the induction hypothesis and find a cut-free derivation of $G\kon C_1\kon C_2\kon C_3\kon C_4\kon E\vdash D$ coded by a term equal to $g\circ(G\kon f\kon E)$.

\vspace{1ex}
(2\emph{b}) If the derivation is of the form
\[
\f{\f{\fp{\mathcal{D}_1}{f\colon C_1\vdash C_2}\quad \fp{\mathcal{D}_2}{g\colon C_3\kon C_4\vdash A}}{C_1\kon (C_2\str C_3)\kon C_4\vdash A}\quad\quad \fp{\mathcal{D}_3}{h\colon G\kon A\kon E\vdash
D}}{G\kon C_1\kon (C_2\str C_3)\kon C_4\kon E\vdash D}
\]
and it is coded by the term
\[
u=h\circ(G\kon(g\circ(\varepsilon\kon C_4)\circ(f\kon(C_2\str
C_3)\kon C_4))\kon E),
\]
then this derivation is transformed into the following derivation
\[
\f{\fp{\mathcal{D}_1}{f\colon C_1\vdash C_2}\quad\quad \f{\f{\fp{\mathcal{D}_2}{g\colon C_3\kon C_4\vdash A}\rzb \fp{\mathcal{D}_3}{h\colon G\kon A\kon E\vdash
D}}{h\circ(G\kon g\kon E)\colon G\kon C_3\kon C_4\kon E\vdash D}}{h\circ(G\kon g\kon E)\circ(\mc\kon C_4\kon E)\colon C_3\kon G\kon C_4\kon E\vdash D}}{\f{v\colon C_1\kon (C_2\str C_3)\kon G\kon C_4\kon E\vdash D}{v\circ(\mc\kon C_4\kon E)\colon G\kon C_1\kon (C_2\str C_3)\kon C_4\kon E\vdash D,}}
\]
for $v$ being the term
\[
h\circ(G\kon g\kon E)\circ(\mc\kon C_4\kon E)\circ(\varepsilon\kon G\kon C_4\kon  E)\circ(f\kon(C_2\str
C_3)\kon G\kon C_4\kon E).
\]

By \ref{nat} and \ref{iso}, the term $v\circ(\mc\kon C_4\kon E)$ is equal to
\[
h\circ(G\kon g\kon E)\circ(G\kon \varepsilon\kon C_4\kon  E)\circ(G\kon f\kon(C_2\str
C_3)\kon C_4\kon E),
\]
and by \ref{fun2}, this term is equal to $u$. Since the rank of the new cut is lower than the rank of the original, and the degrees are the same, we proceed as in (2\emph{a}).

\vspace{1ex}
(2\emph{c}) If the derivation is of the form
\[
\f{\f{\fp{\mathcal{D}_1}{f\colon C_1\vdash A}\quad \fp{\mathcal{D}_2}{g\colon C_2\vdash I}}{f\kon g\colon C_1\kon C_2\vdash A}\quad\quad \fp{\mathcal{D}_3}{h\colon G\kon A\kon E\vdash
D}}{h\circ(G\kon f\kon g \kon E)\colon G\kon C_1\kon C_2\kon E\vdash D},
\]
where $A$ is not $I$, then it is transformed into the following derivation
\[
\f{\fp{\mathcal{D}_2}{g\colon C_2\vdash I}\quad\quad\f{\fp{\mathcal{D}_1}{f\colon C_1\vdash A}\quad \fp{\mathcal{D}_3}{h\colon G\kon A\kon E\vdash
D}}{h\circ(G\kon f\kon E)\colon G\kon C_1\kon E\vdash
D}}{h\circ(G\kon f\kon E)\circ(G\kon C_1\kon g\kon E)\colon G\kon C_1\kon C_2\kon E\vdash D}.
\]

By \ref{fun2}, the terms coding these two derivations are equal. The rank of the upper cut is lower than the rank of the original cut, while the degree remains the same. The degree of the lower cut is less than the degree of the original cut and we may proceed as in (2\emph{a}).

\vspace{1ex}
(2\emph{d}) If the derivation is of the form
\[
\f{\f{\fp{\mathcal{D}_1}{f\colon C_1\vdash I}\quad \fp{\mathcal{D}_2}{g\colon C_2\vdash I}}{f\kon g\colon C_1\kon C_2\vdash I}\quad\quad \fp{\mathcal{D}_3}{h\colon G\kon E\vdash
D}}{h\circ(G\kon f\kon g \kon E)\colon G\kon C_1\kon C_2\kon E\vdash D},
\]
then it is transformed into the following derivation
\[
\f{\fp{\mathcal{D}_2}{g\colon C_2\vdash I}\quad\quad\f{\fp{\mathcal{D}_1}{f\colon C_1\vdash I}\quad \fp{\mathcal{D}_3}{h\colon G\kon E\vdash
D}}{h\circ(G\kon f\kon E)\colon G\kon C_1\kon E\vdash
D}}{h\circ(G\kon f\kon E)\circ(G\kon C_1\kon g\kon E)\colon G\kon C_1\kon C_2\kon E\vdash D},
\]
and again by \ref{fun2}, the terms coding these two derivations are equal. Now both cuts have the ranks lower than the rank of the original cut and we may proceed as in (2\emph{a}).

\vspace{1ex}
(2\emph{e}) The situation presented at the left-hand side of (\ref{5}) has two essentially different cases---one with $A$ as a parameter and another with $A$ involved in the interchange. Let us first consider the case when the derivation is of the form
\[
\f{\fp{\mathcal{D}_1}{f\colon C\vdash
A}\rzb\f{\fp{\mathcal{D}_2}{g\colon G\kon A\kon E_1\kon E_2\kon E_3\kon E_4\vdash D}}{G\kon A\kon E_1\kon E_3\kon E_2\kon E_4\vdash D}}{G\kon C\kon E_1\kon E_3\kon E_2\kon E_4\vdash D,}
\]
and it is coded by the term
\[
u=g\circ(G\kon A\kon E_1\kon\mc\kon E_4)\circ(G\kon f\kon E_1\kon E_3\kon E_2\kon E_4).
\]
This derivation is transformed into the derivation
\[
\f{\f{\fp{\mathcal{D}_1}{f\colon C\vdash
A}\rzb\fp{\mathcal{D}_2}{g\colon G\kon A\kon E_1\kon E_2\kon E_3\kon E_4\vdash D}}{g\circ(G\kon f\kon E_1\kon E_2\kon E_3\kon E_4)\colon G\kon C\kon E_1\kon E_2\kon E_3\kon E_4\vdash D}}{G\kon C\kon E_1\kon E_3\kon E_2\kon E_4\vdash D,}
\]
which is coded by the term
\[
u'=g\circ(G\kon f\kon E_1\kon E_2\kon E_3\kon E_4)\circ (G\kon C\kon E_1\kon \mc\kon E_4).
\]
By \ref{fun2}, the terms $u$ and $u'$ are equal and we proceed as in (2\emph{a}).

\vspace{1ex}
Next, we consider the case when the derivation is of the form
\[
\f{\fp{\mathcal{D}_1}{f\colon C\vdash
A}\rzb\f{\fp{\mathcal{D}_2}{g\colon G_1\kon G_2\kon A\kon G_3\kon E\vdash D}}{g\circ(G_1\kon\mc\kon E)\colon G_1\kon G_3\kon G_2\kon A\kon E\vdash D}}{g\circ(G_1\kon\mc\kon E)\circ(G_1\kon G_3\kon G_2\kon f\kon E)\colon G_1\kon G_3\kon G_2\kon C\kon E\vdash D,}
\]
and it is transformed into
\[
\f{\f{\fp{\mathcal{D}_1}{f\colon C\vdash
A}\rzb\fp{\mathcal{D}_2}{g\colon G_1\kon G_2\kon A\kon G_3\kon E\vdash D}}{g\circ(G_1\kon G_2\kon f\kon G_3\kon E)\colon G_1\kon G_3\kon G_2\kon A\kon E\vdash D}}{g\circ(G_1\kon G_2\kon f\kon G_3\kon E)\circ(G_1\kon\mc\kon E)\colon G_1\kon G_3\kon G_2\kon C\kon E\vdash D.}
\]
By \ref{nat}, the terms coding these two derivations are equal and we proceed as in (2\emph{a}).

\vspace{1ex}
(2\emph{f}) The situation presented at the right-hand side of (\ref{5}) is captured by a derivation of the following form
\[
\f{\fp{\mathcal{D}_1}{f\colon C\vdash
A}\rzb\f{\fp{\mathcal{D}_2}{g\colon G\kon A\vdash X}\rzb \fp{\mathcal{D}_3}{h\colon Y\kon E'\vdash D}}{G\kon A\kon(X\str Y)\kon E'\vdash D}}{G\kon C\kon(X\str Y)\kon E'\vdash D,}
\]
which is coded by the term
\[
h\circ(\varepsilon\kon E')\circ(g\kon(X\str
Y)\kon E')\circ(G\kon f\kon(X\str Y)\kon E').
\]
This derivation is transformed into the derivation
\[
\f{\f{\fp{\mathcal{D}_1}{f\colon C\vdash
A}\rzb \fp{\mathcal{D}_2}{g\colon G\kon A\vdash X}}{g\circ(G\kon f)\colon G\kon C\vdash X}\rzb \fp{\mathcal{D}_3}{h\colon Y\kon E'\vdash D}}{h\!\circ\!(\varepsilon\kon E')\!\circ\!((g\!\circ\!(G\kon f))\kon(X\str Y)\kon E')\colon G\kon C\kon(X\str Y)\kon E'\vdash D,}
\]
and the terms coding these two derivations are equal by \ref{fun2}. Then we proceed as in (2\emph{a}).

\vspace{1ex}
(2\emph{g}) The situation presented at the left-hand side of (\ref{6}) is captured by a derivation of the following form
\[
\f{\fp{\mathcal{D}_1}{f\colon C\vdash
A}\rzb\f{\fp{\mathcal{D}_2}{g\colon D_1\kon G\kon A\kon E\vdash D_2}}{(D_1\str g)\circ \eta\colon G\kon A\kon E\vdash D_1\str D_2}}{(D_1\str g)\circ \eta\circ (G\kon f\kon E)\colon G\kon C\kon E\vdash D_1\str D_2.}
\]
This derivation is transformed into the derivation
\[
\f{\f{\fp{\mathcal{D}_1}{f\colon C\vdash
A}\rzb\fp{\mathcal{D}_2}{g\colon D_1\kon G\kon A\kon E\vdash D_2}}{g\circ (D_1\kon G\kon f\kon E)\colon D_1\kon G\kon C\kon E\vdash D_2}}{(D_1\str (g\circ (D_1\kon G\kon f\kon E)))\circ \eta\colon G\kon C\kon E\vdash D_1\str D_2,}
\]
and the terms coding these two derivations are equal by \ref{nateta} and \ref{fun2str}. Then we proceed as in (2\emph{a}).

\vspace{1ex}
(2\emph{h}) Finally, the situation presented at the right-hand side of (\ref{6}) is captured by a derivation of the following form
\[
\f{\fp{\mathcal{D}_1}{f\colon C\vdash
A}\rzb\f{\fp{\mathcal{D}_2}{g\colon G\kon A\kon E_1\vdash D_1}\rzb \fp{\mathcal{D}_3}{h\colon E_2\vdash D_2}}{g\kon h\colon G\kon A\kon E_1\kon E_2\vdash D_1\kon D_2}}{(g\kon h)\circ (G\kon f\kon E_1\kon E_2)\colon G\kon C\kon E_1\kon E_2\vdash D_1\kon D_2.}
\]
This derivation is transformed into the derivation
\[
\f{\f{\fp{\mathcal{D}_1}{f\colon C\vdash
A}\rzb\fp{\mathcal{D}_2}{g\colon G\kon A\kon E_1\vdash D_1}}{g\circ(G\kon f\kon E_1)\colon G\kon A\kon E_1\kon E_2\vdash D_1\kon D_2}\rzb \fp{\mathcal{D}_3}{h\colon E_2\vdash D_2}}{(g\circ(G\kon f\kon E_1))\kon h\colon G\kon C\kon E_1\kon E_2\vdash D_1\kon D_2,}
\]
and the terms coding these two derivations are equal by \ref{fun2}. Then we proceed as in (2\emph{a}).
\end{proof}

The following lemma serves for the basis of the inductive proof of Proposition~\ref{translation}.

\begin{lem}\label{dodatak1}
If $\Gamma^\ast=B^\ast$, then $\Gamma\vdash B$ has a cut-free derivation in $\mathcal{S}$.
\end{lem}

\begin{proof}
Let $\Gamma$ be $A_1,\ldots A_n$. By Remark~\ref{atomize}, we may assume that no $A_j$ has $\kon$ as the main connective. We proceed by induction on the complexity of $B$ plus the complexity of $\Gamma$ (the number of symbols, including commas, in $B$ and $\Gamma$).

\vspace{1ex}
\noindent(1) [$B$ is $I$] Start with the axiom $\vdash I$ and apply weakening since each $A_j$ must be $I$.

\vspace{1ex}
\noindent(2) [$B$ is $p$] There is $A_j$, which is $p$, and the other members of $\Gamma$ are $I$. Start with the axiom $p\vdash p$ and apply weakening and interchange if necessary.

\vspace{1ex}
\noindent(3) [$B$ is $B_1\kon B_2$] If $B_2^\ast$ is $I$, then apply the induction hypothesis to $\Gamma$ and $B_1$ in order to obtain a cut-free derivation of $\Gamma\vdash B_1$. Then, from the axiom $\vdash I$, using $\vdash\kon$, make a cut-free derivation of $\vdash B_2$. Eventually, apply $\vdash\kon$ to $\Gamma\vdash B_1$ and $\vdash B_2$. The case when $B_1^\ast$ is $I$ is analogous. If neither $B_1^\ast$ nor $B_2^\ast$ is $I$, then by our assumption of the form of $\Gamma$, there must be $1\leq k< n$ such that $(A_1,\ldots,A_k)^\ast=B_1^\ast$ and $(A_{k+1},\ldots,A_n)^\ast=B_2^\ast$. By the induction hypothesis, there are cut-free derivations of $A_1,\ldots,A_k\vdash B_1$ and $A_{k+1},\ldots,A_n\vdash B_2$ and it remains to apply $\vdash\kon$.

\vspace{1ex}
\noindent(4) [$B$ is $B_1\str B_2$] There is $A_j$, which is of the form $C_1\str C_2$ and $C_i^\ast=B_i^\ast$, while the other members of $\Gamma$ are $I$. Apply the induction hypothesis to the singleton sequence $B_1$ and formula $C_1$ in order to obtain a cut-free derivation of $B_1\vdash C_1$ (this case is the reason why the complexity of $\Gamma$ is included). Then apply the induction hypothesis to the singleton sequence $C_2$ and formula $B_2$ in order to obtain a cut-free derivation of $C_2\vdash B_2$. It remains to apply first $\str\vdash$, then $\vdash\str$ and weakening and interchange if necessary.
\end{proof}

The following lemma, which can be taken as a case in the cut-elimination procedure for $\mathcal{S}$ (the one not covered by Theorem~\ref{cutelim1}), is also used in the proof of Proposition~\ref{translation}.

\begin{lem}\label{dodatak2}
If $\Gamma_1\vdash I$ and $\Gamma_2\vdash A$ have cut-free derivations in $\mathcal{S}$, then $\Gamma_1,\Gamma_2\vdash A$ has such a derivation.
\end{lem}

\begin{proof}
Let $\mathcal{D}$ be a cut-free derivation of $\Gamma_1\vdash I$. We proceed by induction on the complexity of $\mathcal{D}$. If $\Gamma_1\vdash I$ is an axiomatic sequent $I\vdash I$ or $\vdash I$, then we either apply weakening, or do nothing to the cut-free derivation of $\Gamma_2\vdash A$. If the last rule in $\mathcal{D}$ is interchange, weakening, $\kon\vdash$ or $\str\vdash$, then we apply the induction hypothesis to the derivation of the upper sequent in the first three cases, or to the derivation of the right upper sequent in the last case, and then apply interchange, weakening, $\kon\vdash$ or $\str\vdash$.
\end{proof}

\begin{proof}[Proof of Proposition~\ref{translation}]
In both directions we proceed by induction on complexity of derivations. For the ``only if'' direction we use the fact that $\ast$ directly translates an inference figure in $\mathcal{S}$ to the corresponding inference figure in $\mathcal{IL}$, save that weakening and $\otimes\vdash$ just vanish in $\mathcal{IL}$.

By Theorem~\ref{cutelim1}, for every $\mathcal{IL}$-derivation there is a cut-free $\mathcal{IL}$-derivation of the same endsequent, thus it is not necessary to treat the cut rule in the ``if'' direction of this proposition. We start with a cut-free derivation $\mathcal{D}$ of $\Gamma^\ast\vdash B^\ast$ in $\mathcal{IL}$, and our goal is to produce a cut-free derivation of $\Gamma\vdash B$ in $\mathcal{S}$.
In the base of the induction $\Gamma^\ast\vdash B^\ast$ is an axiomatic sequent. Hence, $\Gamma^\ast=B^\ast$ and we may apply Lemma~\ref{dodatak1}. In the sequel we use the same assumption on $\Gamma$ as at the beginning of the proof of Lemma~\ref{dodatak1}.

If the last rule in $\mathcal{D}$ is interchange
\[
\f{D\kon E\kon F\kon G\vdash B^\ast}{D\kon F\kon E\kon G\vdash
B^\ast,}
\]
then if either $E$ or $F$ is $I$, this rule has the same upper and lower sequent and one can apply the induction hypothesis. If this is not the case, then by our assumption on $\Gamma$, there are $1\leq i<k\leq l\leq n$ such that $(A_i,\ldots,A_{k-1})^\ast=F$ and $(A_k,\ldots,A_l)^\ast=E$. One can apply the induction hypothesis to the cut-free derivation of $D\kon E\kon F\kon G\vdash B^\ast$, the sequence $A_1,\ldots,A_{i-1},A_k,\ldots,A_l,A_i,\ldots,A_{k-1},A_{l+1},\ldots A_n$, and the formula $B$. It remains to apply (possibly several) interchanges in order to derive $\Gamma\vdash B$.

If the last rule in $\mathcal{D}$ is $\str\vdash$
\[
\f{C\vdash D\quad E\kon F\vdash B^\ast}{C\kon (D\str E)\kon F
\vdash B^\ast},
\]
then by our assumption on $\Gamma$, there is $A_j$ of the form $D'\str E'$ such that $(D')^\ast=D$ and $(E')^\ast=E$. Moreover, we have that $(A_1,\ldots,A_{j-1})^\ast=C$ and $(A_{j+1},\ldots,A_n)^\ast=F$. By applying the induction hypothesis to the cut-free derivation of $C\vdash D$, the sequence $A_1,\ldots,A_{j-1}$ and the formula $D'$, we obtain a cut-free derivation of $A_1,\ldots,A_{j-1}\vdash D'$ in $\mathcal{S}$. In the same way we obtain a cut-free derivation of $E',A_{j+1},\ldots,A_n\vdash B$ and it remains to apply the rule $\str\vdash$ of $\mathcal{S}$.

If the last rule in $\mathcal{D}$ is $\vdash\str$, then  $B$ is of the form $B_1\str B_2$ possibly ``tensored'' by some $I$'s (e.g.\ $B$ is $((I\kon I)\kon (B_1\str B_2))\kon I$), and our $\vdash\str$ is of the form
\[
\f{B_1^\ast\kon G\vdash B_2^\ast}{G\vdash B^\ast.}
\]
Then we apply the induction hypothesis to the cut-free derivation of $B_1^\ast\kon G\vdash B_2^\ast$, the sequence $B_1,A_1,\ldots,A_n$, and the formula $B_2$, in order to obtain a cut-free derivation of $B_1,A_1,\ldots,A_n\vdash B_2$, to which one applies the rule $\vdash\str$ of $\mathcal{S}$. If necessary, it remains to apply $\vdash\kon$ several times using the axiom $\vdash I$ and the derivation of $A_1,\ldots, A_n\vdash B_1\str B_2$.

If the last rule in $\mathcal{D}$ is $\kon\vdash\kon$
\[
\f{C\vdash B_1^\ast\quad D\vdash B_2^\ast}{C\kon D\vdash B^\ast,}
\]
then for some $0\leq k\leq n$, we have that $(A_1,\ldots,A_k)^\ast=C$, $(A_{k+1},\ldots,A_n)^\ast=D$ (note that when $k=0$ the first sequence is empty, while the second is empty when $k=n$). By the induction hypothesis, there are cut-free derivations of $A_1,\ldots,A_k\vdash B_1$ and of $A_{k+1},\ldots,A_n\vdash B_2$ in $\mathcal{S}$.

If neither $B_1^\ast$ nor $B_2^\ast$ is $I$, then $B_1$ and $B_2$ may be chosen so that $B$ is $B_1\kon B_2$, and it remains to apply $\vdash\kon$ to the obtained derivations. If $B_1^\ast$ is $I$, then $B_1,B_2$ can be chosen to be $I$ and $B$, respectively, and we apply Lemma~\ref{dodatak2} in order to obtain a cut-free derivation of $A_1,\ldots,A_n\vdash B$ in $\mathcal{S}$. The case when $B_2^\ast$ is $I$ is analogous up to some interchanges. (Note that it is necessary to treat the case when either $B_1^\ast$ or $B_2^\ast$ is $I$ separately---e.g., let the last rule of an $\mathcal{IL}$-derivation be $\kon\vdash\kon$ with $p\kon (p\str I)\vdash I$ and $q\vdash q$ as upper sequents, while $\Gamma$ is $p,p\str I,q$, and $B$ is $q$.)
\end{proof}

\begin{center}\textmd{\textbf{Acknowledgements}}
\end{center}
\smallskip
The authors were supported by the Science Fund of the Republic of Serbia, Grant No. 7749891, Graphical Languages - GWORDS. We are sincerely grateful to the reviewers for their insightful comments and essential recommendations that significantly improved the clarity of this work.


\medskip

\begin{thebibliography}{99}

\bibitem{BFSV} {\sc C.\ Balteanu, Z.\ Fiedorowicz, R.\ Schw\" anzl} and {\sc R.\ Vogt}, {\it Iterated monoidal categories}, \textbf{\textit{Advances in Mathematics}}, vol.\ 176 (2003), pp.\ 277-349

\bibitem{B11} {\sc W. Burnside}, \textbf{\textit{Theory of Groups of Finite Order}}, second edition, Cambridge University Press, Cambridge, 1911 (reprint, Dover, New York, 1955)

\bibitem{D03} {\sc K.\ Do\v sen}, {\it Identity of proofs based on normalization and generality}, \textbf{\textit{The Bulletin of Symbolic Logic}}, vol.\ 9 (2003), pp.\ 477-503

\bibitem{DP04} {\sc K.\ Do\v sen} and {\sc Z.\ Petri\' c}, \textbf{\textit{Proof-Theoretical Coherence}}, KCL Publications, London, 2004

\bibitem{DP07} --------, \textbf{\textit{Proof-Net Categories}}, Polimetrica, Monza, 2007

\bibitem{K93} {\sc M.\ Kapranov}, {\it The permutoassociahedron, Mac Lane's coherence theorem and asymptotic zones for the KZ equation}, \textbf{\textit{Journal of Pure and Applied Algebra}}, vol.\ 85 (1993) pp.\ 119-142

\bibitem{KML71} {\sc G.M.\ Kelly} and {\sc S.\ Mac Lane}, {\it Coherence in closed categories}, \textbf{\textit{Journal of Pure and Applied Algebra}}, vol.\ 1 (1971), pp.\ 97-140

\bibitem{K03} {\sc J.\ Kock}, \textbf{\textit{Frobenius Algebras and 2D Topological Quantum Field Theories}}, Cambridge University Press, Cambridge, 2003

\bibitem{L68} {\sc J.\ Lambek}, {\it Deductive systems and categories I: Syntactic calculus and residuated categories}, \textbf{\textit{Mathematical Systems Theory}}, vol.\ 2 (1968), pp.\ 287-318

\bibitem{L69} --------, {\it Deductive systems and categories II: Standard constructions and closed categories}, \textbf{\textit{Category Theory, Homology Theory and their Applications I}}, Lecture Notes in Mathematics, vol.\ 86, Springer, Berlin, 1969, pp.\ 76-122

\bibitem{LS86} {\sc J. Lambek} and {\sc P.J. Scott}, \textbf{\textit{Introduction to Higher Order Categorical Logic}}, Cambridge University Press, Cambridge, 1986

\bibitem{ML63} {\sc S.\ Mac Lane}, {\it Natural associativity and commutativity}, \textbf{\textit{Rice University Studies, Papers in Mathematics}}, vol.\ 49 (1963), pp.\ 28-46

\bibitem{ML98} --------, \textbf{\textit{Categories for the Working Mathematician}}, second edition, Springer, Berlin, 1998

\bibitem{MS07} {\sc L. M\' ehats} and {\sc S.V. Soloviev}, {\it Coherence in SMCCs and equivalences on derivations in IMLL with unit}, \textbf{\textit{Annals of Pure and Applied Logic}}, vol. 147 (2007), pp. 127-179

\bibitem{M896} {\sc E.H.\ Moore}, {\it Concerning the abstract groups of order $k!$ and $\frac{1}{2} k!$ holohedrically isomorphic with the symmetric and the alternating substitution-groups on $k$ letters}, \textbf{\textit{Proceedings of the London Mathematical Society}}, vol.\ s1-28 (1896), pp.\ 357-366

\bibitem{S09} {\sc P.\ Selinger}, {\it A survey of graphical languages for monoidal categories}, (B.\ Coecke editor) \textbf{\textit{New Structures for Physics}}, Lecture Notes in Physics 813, Springer, 2009, pp.\ 289-355

\bibitem{S90} {\sc S.V. Soloviev}, {\it On the conditions of full coherence in closed categories}, \textbf{\textit{Journal of Pure and Applied Algebra}}, vol. 69 (1990), pp. 301-329 (Russian version in  \textbf{\textit{Matematicheskie metody postroeniya i analiza algoritmov}}, A.O. Slisenko and S.V. Soloviev, editors, Nauka, Leningrad, 1990, pp. 163-189)

\bibitem{S97} --------, {\it Proof of a conjecture of S. Mac Lane}, \textbf{\textit{Annals of Pure and Applied Logic}}, vol. 90 (1997), pp. 101-162

\bibitem{S63} {\sc J.D.\ Stasheff}, {\it Homotopy associativity of H-spaces, I, II}, \textbf{\textit{Transactions of the American Mathematical Society}}, vol.\ 108 (1963), pp.\ 275-292, 293-312

\bibitem{T10} {\sc V.G.\ Turaev}, \textbf{\textit{Quantum Invariants of Knots and 3-Manifolds}}, de Gruyter Studies in Mathematics 18, De Gruyter, 2010

\bibitem{V77} {\sc R.\ Voreadou}, \textbf{\textit{Coherence and Non-Commutative Diagrams in Closed Categories}}, Memoirs of the American Mathematical Society, no. 182, American Mathematical Society, 1977


\end{thebibliography}
\end{document}